\documentclass[12pt]{amsart}
\usepackage{amsmath,amssymb,amsfonts,amsthm,amsopn}
\usepackage{latexsym,graphicx}
\usepackage{xcolor}
\usepackage{color, colortbl}
\usepackage{enumitem}
\usepackage{pb-diagram}
\usepackage[title]{appendix}
\usepackage{mathtools}

\mathtoolsset{showonlyrefs}

\newtheorem{theorem}{Theorem}[section]

\newtheorem{corollary}[theorem]{Corollary}
\newtheorem{proposition}[theorem]{Proposition}
\newtheorem{definition}[theorem]{Definition}

\newtheorem{remark}[theorem]{Remark}

\newcommand{\beqa}{\begin{eqnarray*}}
\newcommand{\eeqa}{\end{eqnarray*}}

\DeclareMathOperator*{\Airy}{\mbox{Ai}}

\newcommand{\field}[1]{\mathbb{#1}}
\newcommand{\bR}{\field{R}}        
\newcommand{\bN}{\field{N}}        
\newcommand{\bZ}{\field{Z}}        
\newcommand{\bC}{\field{C}}        
\def\la{\lambda}

\def\eps{\epsilon}

\def\cF{\mathcal{F}}              
\def\cS{\mathcal{S}}

\def\cG{\mathcal{G}}

\def\cC{\mathcal{C}}

\def\a{\aleph}

\def\rd{\bR^d}

\def\rdd{\bR^{2d}}

\def\intrd{\int_{\rd}}
\def\intrdd{\int_{\rdd}}

\def\l{\langle}

\def\<{\left<}
\def\>{\right>}

\def\mv1{M_v^1}

\def\phas{(x,\xi )}
\def\mn{(m,n)}
\def\mn'{(m',n')}

\def\o{\xi}
\def\a{\alpha}
\def\b{\beta}

\def\Ren{\mathbb{R}^d}

\def\f{\varphi}

\def\Sn2{S_{2}(L^{2}(\Ren))}
\def\S1{S_{1}(L^{2}(\Ren))}
\def\sig00{\sigma_{0,0}}

\def\la{\langle}
\def\ra{\rangle}

\begin{document}

\begin{abstract} 
The integration of operator kernels with the Wigner distribution, first conceptualized by E. Wigner in 1932 and later extended by L. Cohen and others, has opened new avenues in time-frequency analysis and operator calculus. Despite substantial advancements, the presence of ``ghost frequencies" in Wigner kernels continues to pose significant challenges, particularly in the analysis of Fourier integral operators (FIOs) and their applications to partial differential equations (PDEs). 

In this work, we build on the foundational concepts of Wigner analysis to introduce a novel framework for controlling ghost frequencies through the combined use of Gaussian and Sobolev regularization techniques. By focusing on FIOs with non-quadratic phase functions, we develop rigorous estimates for the Wigner kernels that are crucial for their applicability to Schrödinger equations with non-trivial symbol classes. Unlike previous approaches, our methodology not only mitigates the interference caused by ghost frequencies but also establishes robust bounds in the context of generalized symplectic mappings.
\end{abstract}

\title[Wigner analysis of operators. Part III]{Wigner analysis of operators. Part III: Controlling ghost frequencies}

\author{Elena Cordero}
\address{Universit\`a di Torino, Dipartimento di Matematica, via Carlo Alberto 10, 10123 Torino, Italy}
\email{elena.cordero@unito.it}

\author{Gianluca Giacchi}
\address{Università della Svizzera Italiana, Faculty of Informatics, Via la Santa 1, 6962 Viganello, Switzerland; Università di Bologna, Dipartimento di Matematica, Piazza di Porta San Donato 5, 40126 Bologna, Italy; University of Lausanne, Switzerland; HES-SO School of Engineering, Rue De L'Industrie 21, Sion, Switzerland; Centre Hospitalier Universitaire Vaudois, Switzerland}
\email{giaccg@usi.ch}

\author{Luigi Rodino}
\address{Universit\`a di Torino, Dipartimento di Matematica, via Carlo Alberto 10, 10123 Torino, Italy}
\email{luigi.rodino@unito.it}

\thanks{}

\maketitle

\section{Introduction} 
The original idea of combining operator kernels with the Wigner distribution, pioneered by E. Wigner in 1932 \cite{Wigner1932} and pursued by  J.G. Kirkwood \cite{Kirkwood}, laid the groundwork for a novel approach in operator calculus. Following Wigner's idea, L. Cohen and L. Galleani \cite{GalleaniCohen} applied the Wigner distribution  to classical systems   when the governing equation of the variable is a linear ordinary or partial differential equation. 
Recently, C. Mele and A. Oliaro \cite{MO} studied  regularity of partial differential equations with polynomial coefficients  via Wigner distribution  and proved regularity properties for these classes.

 While Wigner's initial formulation provided a conceptual framework, its refinement and rigorous formulation have been recently achieved by the authors in \cite{CGR2024-1, CGR2024,  CGR2024-2, CGRV2024,CR2022}.

The protagonist of this study is the time-frequency representation named cross-Wigner distribution after Wigner.
Precisely, for $f,g\in L^2(\rd)$, the cross-Wigner distribution $W(f,g)$  is
\begin{equation}\label{intro1}
	W(f,g)(x,\xi)=\int_{\rd}f(x+t/2)\overline{g(x-t/2)}e^{-2\pi i\xi\cdot t}dt, \qquad x,\xi\in\rd.
\end{equation}
The Wigner distribution of $f\in L^2(\rd)$ is $Wf=W(f,f)$.

For a given linear continuous operator $T:\cS(\rd)\to\cS'(\rd)$, Wigner's initial approach consisted in considering another operator $K:\cS(\rdd)\to\cS'(\rdd)$  whose kernel $k_W$ acted as a Schwartz kernel \emph{with respect to the Wigner distribution}. Specifically, 
\begin{equation}\label{intro2}
	W(Tf)(z)=K(Wf)(z)=\int_{\rdd}k_W(z,w)Wf(w)dw, \qquad f\in\cS(\rd).
\end{equation}
However, a precise definition of such an operator necessitates the use of the cross-Wigner distribution \eqref{intro1}. The \emph{Wigner kernel of $T$} is defined as the Schwartz kernel $k_W$ of the operator $K:\cS(\rdd)\to\cS'(\rdd)$ satisfying the intertwining relation:
\begin{equation}\label{intro3}
	W(Tf,Tg)(z)=K(W(f,g))(z)=\int_{\rdd}k_W(z,w)W(f,g)(w)dw,
\end{equation}
for every $f,g\in\cS(\rd)$.

In the present work, we consider integral operators in the form
\begin{equation}\label{intro4}
	T_If(x)=\int_{\rd}e^{2\pi i\Phi(x,\xi)}\sigma(x,\xi)\hat f(\xi)d\xi, \qquad f\in\cS(\rd),
\end{equation}
with phase $\Phi$ and symbol $\sigma$ in suitable classes.  This operators are known as Fourier integral operators (FIO) of Type I \cite{B36,treves} and have been extensively studied in the literature, cf. \cite{CGNRJMPA,coriasco1,coriasco,coriascoruz,Guo-labate,concetti-toft} for a very partial list of contributions. We look for estimates of the type
\begin{equation}\label{intro5}
	|k_W(z,w)|\lesssim\la z-\chi(w)\ra^{2N},
\end{equation}
where $\la \cdot\ra=(1+|\cdot|^2)^{1/2}$, $N\in\bZ$, and $\chi$ is the symplectic map related to the phase $\Phi$ by solving
\begin{equation}\label{intro6}
	\begin{cases}
		y=\Phi_\eta(x,\eta),\\
		\xi=\Phi_x(x,\eta),
	\end{cases}
\end{equation}
 with respect to $\phas$, so that   $z=\chi(w)$, with $z=\phas$ and $w=(y,\eta)$. Estimates \eqref{intro5} are obtained in \cite{CGRV2024} by assuming the symbol $\sigma$ is in Shubin classes of sufficiently negative order. In this work, we address to general symbols suitable for applications to Schr\"odinger equations. In this general scenario, two obstacles arise regarding the validity of \eqref{intro5}. Firstly, the point-wise definition of the Wigner kernel $k_W(z,w)$ for $T$ in \eqref{intro3} may not hold for $z=\chi(w)$. This challenge can be effectively addressed through a rescaling regularization, achieved by convolving with a Bessel-Sobolev potential. Secondly, a more fundamental hurdle arises concerning the nonlinear symplectic map $\chi$. It is widely known that the Wigner distribution introduces ghost frequencies, which contradict \eqref{intro5}, see the contributions of P. Boggiatto, E. Carypis, G. de Donno and A. Oliaro \cite{bcoquadratic,BdDO4,BdDO3,BdDO2,BDO2007}. While a Sobolev regularization can mitigate the impact of ghost frequencies up to a certain order $N$, their complete elimination necessitates convolution with a Gaussian function. Consequently, it becomes imperative to replace $k_W$ in \eqref{intro5} with a smoothed kernel, as clarified in the sequel.

The function spaces for our study are given by the H\"ormander class
$ S^0_{0,0}(\bR^{n})$ \cite{B36} and the exotic class $E(\bR^n\times\bR^m)$ in Definition \ref{defe1}.
Recall that   $S^0_{0,0}(\bR^{n})$ is the class of complex functions $\sigma\in\cC^\infty(\bR^n)$ such that for every $\alpha\in\bN^n$ there exists a constant $C_\alpha>0$ with
\begin{equation}\label{intro7}
	|\partial^\alpha\sigma(z)|\leq C_\alpha, \qquad  \ z\in\bR^n.
\end{equation}
\begin{definition}[The class $E(\bR^n\times\bR^m)$]\label{defe1}
A function $\sigma:\bR^n\times\bR^m\to \bC$ belongs to the class $E(\bR^n\times\bR^m)$,  if $\sigma\in\cC^\infty(\bR^n\times\bR^m)$ and  for every $\alpha\in\bN^n,\beta\in\bN^m$,
\begin{equation}\label{intro8}
	|\partial^\alpha_u\partial^\beta_v\sigma(u,v)|\leq C_{\alpha,\beta}\la v\ra^{3(|\a|+|\beta|)}, \qquad (u,v)\in \bR^n\times\bR^m.
\end{equation} 
\end{definition}
Our study will encompass type I FIOs with \emph{tame} phases (see Definition \ref{defe2} below), as appeared in many papers in the literature, cf. e.g., \cite{concetti-garello-toft, CGNRJMPA,ruzhsugimoto,concetti-toft}.

First, we will exhibit the relation between the Wigner kernel $k_W$ of a type I FIO $T_I$, with a  phase function $\Phi$ which extends the tame case above, and its symbol $\sigma$. 
\begin{proposition}\label{prop-intro1}
Consider the type I FIO in \eqref{intro4} with symbol $\sigma\in S^0_{0,0}(\rd\times\rd)$ and a real phase $\Phi\in\cC^\infty(\rdd)$ such that
	\begin{equation}\label{intro9}
		\partial^\alpha_x\partial^\beta_\xi\Phi(x,\xi)\in S^0_{0,0}(\rd\times\rd), \qquad |\alpha|+|\beta|\geq3.
	\end{equation}
	Then, the Wigner kernel $k_W(z,w)$ of $T_I$ is given by:
	\begin{equation}\label{intro10}
		k_W(x,\xi,y,\eta)=\int_{\rdd}e^{-2\pi i [t\cdot(\xi-\Phi_x(x,\eta))+r\cdot(y-\Phi_\eta(x,\eta))]}\tilde\sigma(x,\eta,t,r)dtdr,
	\end{equation}
where
\begin{equation}\label{sigmatilde}
	\tilde{\sigma}(x,\eta,t,r)=e^{2\pi
		i[\Phi_{3}-\widetilde{\Phi}_{3}](x,\eta,t,r)} \sigma(x+\frac{t}{2},\eta+\frac{r}{2})\overline{\sigma(x-\frac{t}{2},\eta-\frac{r}{2})},
\end{equation}
with
 \begin{equation}
	\label{eq:c11}
	\Phi_{3}(x,\eta,t,r)=\sum_{|\a|=3}\int_0^1(1-\tau)\partial^\a
	\Phi(\phas+\tau(t,r)/2)\,d\tau\frac{(t,r)^\a}{2^4\a!}.
\end{equation}
and 
\begin{equation}
	\label{eq:c12}
	\widetilde{\Phi}_{3}(x,\eta,t,r)=\sum_{|\a|=3}\int_0^1(1-\tau)\partial^\a
	\Phi(\phas-\tau(t,r)/2)\,d\tau\frac{(t,r)^\a}{2^4\a!}.
\end{equation}
The symbol $\tilde\sigma$ is in the class $E(\bR^{2d}\times \bR^{2d})$. In particular, if  $\Phi$ is a quadratic form then $\tilde\sigma\in S^0_{0,0}(\bR^{4d})$.
\end{proposition}
For quadratic phases  $\Phi$ the result was proved in \cite{CGR2024}. A relevant example is $\Phi(x,\xi)=x\cdot\xi$, i.e., $T_I$ in \eqref{intro4} is a pseudo-differential operator with Kohn-Nirenberg symbol $\sigma(x,\xi)$. Then, $k_W$ is the kernel of a pseudo-differential operator with symbol in dimension $4d$, as already observed in \cite{CR2022}. For non-quadratic phases $\Phi$ the appearance of symbols $\tilde\sigma\in E(\rdd\times\rdd)$ in \eqref{intro4} cannot be avoided. As simple example consider, for $d=1$ in \eqref{intro4}, the phase
\begin{equation}\label{introAMex}
	\Phi(x,\xi)=x\xi-\xi^3,\qquad x,\xi\in\bR,
\end{equation}
and $\sigma=1$. From the proof of Proposition \ref{prop-intro1} below, we will obtain
\[
	\tilde\sigma(x,\eta,t,r)=e^{-i\pi r^3/2}\in E(\bR^2_{x,\eta}\times\bR^2_{t,r}).
\]
To overcome this issue, we shall consider the following two types of smoothing for the Wigner kernel $k_W$. First, define:
\begin{equation}\label{intro11}
	k_{W,G}(z,w):=[k_W\ast G](z,w),
\end{equation}
where $G$ is the Gaussian function on $\bR^{4d}$:
\begin{equation}\label{intro12}
	G(z,w)=e^{-2\pi(z^2+w^2)}, \qquad z,w\in\rdd.
\end{equation}
For $M\in\bZ$ we define the Fourier multiplier
\begin{equation}\label{intro13}
	\la D_x\ra^Mf(x):=\int_{\rd}e^{2\pi ix\cdot \omega}\la \omega\ra^M\hat f(\omega)d\omega=v_M\ast f(x), \qquad x\in\rd, 
\end{equation}
where
\begin{equation}\label{vM}
	v_M(x)=\int_{\rd}e^{2\pi i x\cdot\omega}\la \omega\ra^Md\omega, \qquad x\in\rd,
\end{equation} 
 is the so-called \emph{Bessel-Sobolev potential}.

As second smoothing we consider, for $M\in\bZ$, $M\leq 0$, 
\begin{equation}\label{intro14}
	k_{W,M}(z,w):=\la D_{z,w}\ra^M k_W(z,w),
\end{equation}
where we set
\begin{equation}\label{intro15}
	\la D_{z,w}\ra^M:=\la D_x\ra^M\la D_\xi\ra^M\la D_y\ra^M\la D_\eta\ra^M, \qquad z=(x,\xi), \ w=(y,\eta).
\end{equation}
In order to obtain estimates of type \eqref{intro5}, the Bessel-Sobolev smoothing \eqref{intro14}, \eqref{intro15} is sufficient in the case of a quadratic $\Phi$. In this context, the pseudo-differential case reads as follows.
\begin{theorem}\label{intro-thm2}
	Assume $\Phi(x,\xi)=x\cdot\xi$ in \eqref{intro4}, i.e., $T_I$ is a pseudo-differential operator. If the symbol $\sigma$ is in $S^0_{0,0}(\rd\times\rd)$, then there exists $M\leq0$, depending only on the dimension $d$, such that:
	\begin{equation}\label{intro17}
		|	k_{W,M}(z,w)|\lesssim\la z-w\ra^{-2N}, \qquad \forall N\in\bN.
	\end{equation}
\end{theorem}
For non-quadratic $\Phi$, the estimates \eqref{intro17} fail for large $N$, because $\tilde\sigma\in E(\rdd_{x,\eta}\times\rdd_{t,r})$, see Section \ref{sec3} below for the aforementioned example with $\Phi$ as in \eqref{introAMex} and $\sigma=1$. The Gaussian smoothing \eqref{intro11} and \eqref{intro12} is then needed. 
%
Also, we need some assumptions on $\Phi(x,\eta)$ to derive from \eqref{intro6} the expression of the symplectic  map $\chi$. One possibility is then consider a phase $\Phi$  tame, cf. Definition \ref{defe2} below. 
Otherwise, we may assume $\Phi$ of the form 
\begin{equation}\label{introbis} 
	\Phi(x,\xi)=x\xi-\f(\xi),\quad \partial^\beta_\xi\f\in S^0_{0,0}(\rd),\,\mbox{for}\,\,|\beta|\geq3.
\end{equation}
 In both cases condition \eqref{intro9} is satisfied and $z=\chi(w)$, $z=(x,\xi)$, \, $w=(y,\eta)$, from \eqref{intro6} is well defined, see Proposition 5.2 below for details.  The canonical transformation  $\chi$ associated with $\Phi$ in \eqref{introAMex} is
\begin{equation}\label{intro21}
	z=\chi(w)=(y+3\eta^2,\eta), \qquad z=(x,\xi), \ w=(y,\eta).
\end{equation}

\begin{theorem}\label{intro-thm3}
If the phase $\phi$ is tame or is of the form \eqref{introbis}, the symbol $\sigma$ is in  $S^0_{0,0}(\rd\times\rd)$, then the Gaussian smoothed Wigner kernel $k_{W,G}$ in \eqref{intro11} satisfies the estimate
\begin{equation}\label{intro22}
	|k_{W,G}(z,w)|\lesssim\la z-\chi(w)\ra^{-2N}, \qquad \forall N\in\bN.
\end{equation}
\end{theorem}

As a byproduct, our results extend the existing literature on time-frequency representations by generalizing classical estimates for Gabor matrices and uncovering new regularity properties for Wigner kernels in higher dimensions.  
\begin{definition}[Gabor matrix]
Define the modulation $M_\xi$ and translation $T_x$ operators by
$$T_xf=f(t-x),\quad M_\xi f(t)=e^{2\pi i\xi\cdot t}f(t-x), \quad t, x,\xi\in\rd.$$ Their composition is the so-called  time-frequency shift
\begin{equation}\label{defTFshift}\pi(z)f(t)=M_\xi T_xf(t)=e^{2\pi i\xi\cdot t}f(t-x),\quad z=(x,\xi)\in\rdd.
\end{equation}
Fix two \emph{window functions} $g,\gamma\in\cS(\rd)\setminus\{0\}$. The \emph{Gabor matrix} of a continuous linear operator $T:\cS(\rd)\to\cS'(\rd)$ is defined by
\begin{equation}\label{intro23}
	\cG_m (z,w):=\la T\pi(z)g,\pi(w)\gamma\ra, \qquad z,w\in\rdd.
\end{equation}
\end{definition}
 Fixing
\begin{equation}\label{intro24}
	g(t)=\gamma(t)=2^{-d/4}e^{-\pi t^2},
\end{equation}
we relate the Gabor matrix $\cG_m $ to the Wigner kernel $k_W$:
\begin{equation}\label{intro25}
	|\cG_m (z,w)|^2=[k_W\ast G](z,w)=k_{W,G}(z,w),
\end{equation}
with the notation \eqref{intro11} and \eqref{intro12}, see \cite{CGR2024-2} and Section \ref{sec3} below. Combining the equality in \eqref{intro25} with the estimate for the Gaussian smoothed Wigner kernel $k_{W,G}$  in  \eqref{intro22}, we obtain the following consquence. 
\begin{corollary}\label{intro-cor4}
	Assume $\Phi$ is a tame function, c.f. Definition \ref{defe2}, and  the symbol  $\sigma\in S^0_{0,0}(\rd\times\rd)$. Then, the Gabor matrix $\cG_m$ of $T$ in \eqref{intro25} satisfies
	\[
		|\cG_m(z,w)|\lesssim\la z-\chi(w)\ra^{-N}, \qquad \forall N\in\bN,
	\]
	where $\chi$ is the canonical transformation related to $\Phi$.
\end{corollary}
The above result was first proved in \cite{CGNRJMPA}. The same estimates are valid for the Gabor matrix of the prototypical example of the operator in \eqref{intro4}, with $\sigma=1$, $\Phi$  in \eqref{introAMex}, and $\chi$ in \eqref{intro21}.

To validate the importance of our approach, in Appendix \ref{AppA} we provide examples of how Gaussian convolution reduces ghost frequencies, comparing this smoothing with the Bessel-Sobolev one.

Overall, our findings not only advance the theoretical framework of Wigner analysis but they might have far-reaching implications for practical applications in quantum mechanics and signal processing.

The organization of this paper is as follows. We start with a preliminary
section (Section $2$) devoted to the definition and basic properties of the Wigner distribution and related time-frequency representation, the rigorous definition of the Wigner kernel and main properties, the Fourier integral operators object of our study.  Section $3$ is devoted to the study of the Wigner kernel, in particular we shall prove Theorems \ref{intro-thm2} and \ref{intro-thm3}. Section $4$ gives applications of the above theory to Schr\"{o}dinger equations. Appendix \ref{AppA} exhibits examples of how our smoothing methods reduce ghost frequencies.
\section{Preliminaries}
We denote by $x\xi=x\cdot\xi$, $x,\xi\in\rd$, the standard inner product in $\rd$. If $f,g\in L^2(\rd)$, $\la f,g\ra=\int_{\rd}f(x)\overline{g(x)}dx$ denotes their inner product in $L^2(\rd)$, whereas the same notation is used for the duality pairing $\cS'\times\cS$ (conjugate-linear in the second component). We denote by $\cC(\rd)$ the space of complex-valued continuous functions on $\rd$. To facilitate the reading, we shall adopt the weak integral notation. 

\subsection{Wigner distribution}
	For $f,g\in L^2(\rd)$,	the (cross-)Wigner distribution $W(f,g)$
is defined in \eqref{intro1}. If $f=g$, $Wf=W(f,f)$ denotes the Wigner distribution of $f$. The (cross-)Wigner distribution extends uniquely to $\cS'(\rd)$ as follows:
	\begin{equation}\label{defWSp}
		W(f,g)=\cF_2\mathfrak{T}_w(f\otimes\bar g), \qquad f,g\in\cS'(\rd),
	\end{equation}
	where
	\[
		\cF_2F(x,\xi)=\int_{\rd}F(x,t)e^{-2\pi it\xi}dt, \qquad F\in\cS(\rdd), \quad x,\xi\in\rd,
	\]
	denotes the Fourier transform with respect to the frequency variables, and
	\[
		\mathfrak{T}_wF(x,t)=F(x+t/2,x-t/2), \qquad F\in\cS(\rdd).
	\]
	Observe that $\mathfrak{T}_w^{-1}F(x,y)=F(\frac{x+y}{2},x-y)$. Specifically, \eqref{defWSp} means that $W(f,g)$ is the tempered distribution in $\cS'(\rdd)$ so that
	\[
		\la W(f,g),\Phi\ra=\la f\otimes\bar g,\mathfrak{T}_w^{-1}\cF_2^{-1}\Phi\ra, \qquad \Phi\in\cS(\rdd).
	\]
	Recall that $W:\cS(\rd)\times\cS(\rd)\to\cS(\rdd)$ and $\mbox{span}\{W(f,g):f,g\in\cS(\rd)\}$ is dense in $\cS(\rdd)$. We will use the polarization formula:
	\begin{equation}\label{polarization}
		W(f+g)=Wf+Wg+2\Re(W(f,g)), \qquad f,g\in\cS'(\rd).
	\end{equation}
	
\subsection{The Husimi distribution} This distribution will be used in Appendix \ref{AppA} to provide examples which validate this study.
	Let us consider the Gaussian window $\f(t)=e^{-\pi|t|^2}$. Let $f\in \cS'(\rd)$. The \emph{Husimi distribution} of $f$ is:
	\[
		Hf(x,\xi)=W\f\ast Wf\phas,
	\] 
	where $W\f(x,\xi)=2^{d/2}e^{-2\pi(|x|^2+|\xi|^2)}$. To compute $Hf$ it may be useful to use the following characterization of the Husimi distribution by means of the short-time Fourier transform (STFT):
	\[
		V_\f f(x,\xi)=\la f,\pi(x,\xi)g\ra, \qquad f\in\cS'(\rd),\quad x,\xi\in\rd,
	\]
	where the time-frequency shift $\pi(x,\xi)$ is defined as in \eqref{defTFshift}. In fact,
	\begin{equation}\label{husimi}
		Hf(x,\xi)=|V_\f f(x,\xi)|^2.
	\end{equation}
	We will also use the fundamental identity of time-frequency analysis:
	\begin{equation}\label{fundidTF}
		V_gf(x,\xi)=e^{-2\pi i\xi x}V_{\hat g}\hat f(\xi,-x), \qquad f\in\cS'(\rd), \quad g\in\cS(\rd), \quad x,\xi\in\rd.
	\end{equation}

\subsection{Wigner kernels}
	Let $T:\cS(\rd)\to\cS'(\rd)$ be a linear and continuous operator. It was proved in Theorem 3.4 of \cite{CGR2024-1} that there exists a unique distribution $k_W\in\cS'(\bR^{4d})$ such that
	\begin{equation}\label{def-Wk}
		W(Tf,Tg)(z)=\int_{\rdd}k_W(z,w)W(f,g)(w)dw, \qquad f,g\in\cS(\rd),
	\end{equation} 
	where the integral is understood in the sense of distributions.
	The distribution $k_W$ is called the \emph{Wigner kernel} of $T$. Namely, \eqref{def-Wk} means that $k_W$ is the tempered distribution of $\cS'(\bR^{4d})$ characterized by:
	\[
		\la k_W,\Phi\otimes \overline{W(f,g)}\ra=\la W(Tf,Tg),\Phi\ra, \qquad \Phi\in\cS(\rdd), \quad f,g\in\cS(\rd).
	\]
	
	If $k_T\in\cS'(\rdd)$ is the Schwartz kernel of $T$, i.e., the unique tempered distribution satisfying:
	\begin{equation}
		Tf(x)=\int_{\rd}k_T(x,y)f(y)dy, \qquad f\in\cS(\rd),
	\end{equation}
	then (\cite[Theorem 3.4]{CGR2024-1})
	\[
		k_W(x,\xi,y,\eta)=Wk_T(x,y,\xi,-\eta).
	\]

\subsection{Fourier integral operators}
We first recall the definition of \emph{tame} phases.
\begin{definition}[Tame phases]\label{defe2}
	A phase function $\Phi:\rd\times\rd\to \bR$ is called \emph{tame} if $\Phi\in\cC^\infty(\rdd)$, there exists $\delta>0$ such that
	\begin{equation}\label{tame0}
		|\det(\partial^2_{x,\xi}\Phi(x,\xi))|\geq\delta, 
	\end{equation}
	and
	\begin{equation}\label{tame}
		\partial^\alpha_x\partial^\beta_\xi\Phi(x,\xi)\in S^0_{0,0}(\rd\times\rd), \qquad |\alpha|+|\beta|\geq2.
	\end{equation}
\end{definition}

Consider a symbol $\sigma\in\cS'(\rdd)$. 
The Fourier integral operator (FIO) of type I with symbol $\sigma$ and tame phase $\Phi$ is the operator:
\begin{equation}\label{type1}
	T_If(x)=\int_{\rd}e^{2\pi i\Phi(x,\xi)}\sigma(x,\xi)\hat f(\xi)d\xi, \qquad f\in\cS(\rd).
\end{equation}
A pseudo-differential operator in the Kohn-Nirenberg form is a FIO with phase $\Phi(x,\xi)=x\xi$, i.e. in the form:
\begin{equation}\label{K-N}
	\sigma(x,D)f(x)=\int_{\rd}e^{2\pi i\xi x}\sigma(x,\xi)\hat f(\xi)d\xi, \qquad f\in\cS(\rd).
\end{equation}

\section{Estimates of the Wigner kernel}\label{sec3}
In what follow we exhibit the Wigner kernel  $k_W$ of a type I FIO as in \eqref{type1},  obtained in \cite[Theorem 5.8]{CGR2024-1}.
\begin{theorem}\label{teor3.6}
Let $T_I$ be a FIO of type I as in \eqref{type1} with symbol $\sigma\in S^0_{0,0}(\rdd)$. For $f\in\cS(\rd)$,
	\begin{equation}\label{K}
		K(W(f,g))(x,\xi)=W(T_If,T_Ig)(x,\xi)=\int_{\rdd}k_W(x,\xi,y,\eta)W(f,g)(y,\eta)dyd\eta,
	\end{equation}
	where the Wigner kernel $k_W$ is given by
	\begin{equation}\label{KI}
		k_W(x,\xi,y,\eta)=\int_{\rdd}e^{2\pi i[\Phi_I(x,\eta,t,r)-(\xi t+ry)]} \sigma_I(x,\eta,t,r)dtdr,
	\end{equation}
	with, for $x,\eta,t,r\in\rd$,
	\begin{equation}\label{fii}
		\Phi_I(x,\eta,t,r)=\Phi(x+\frac{t}{2},\eta+\frac{r}{2})-\Phi(x-\frac{t}{2},\eta-\frac{r}{2})
	\end{equation}
	and
	\begin{equation}\label{sigmai}
		\sigma_I(x,\eta,t,r):=\sigma(x+\frac{t}{2},\eta+\frac{r}{2})\overline{\sigma(x-\frac{t}{2},\eta-\frac{r}{2})}.
	\end{equation}
\end{theorem}
This relation is the key tool for what follows.
\begin{proof}[Proof of Proposition \ref{prop-intro1}]\label{sec:prop-intro1proof}
We use the proof's pattern of \cite[Theorem 5.9]{CGR2024-1}. 
The phase  $\Phi$ is smooth and we can expand $\Phi(x+\frac{t}{2},\eta+\frac{r}{2})$ and $\Phi(x-\frac{t}{2},\eta-\frac{r}{2})$
into a Taylor series of second order around
$(x,\eta)$. Namely,
\begin{equation}\label{E1}
	\Phi\left(x+\frac{t}{2},\eta+\frac{r}{2}\right)=\Phi\phas+\frac {t}2\Phi_x \phas+\frac {r}2 \Phi_\eta \phas +\Phi_{2}(x,\eta,t,r)+\Phi_{3}(x,\eta,t,r),
\end{equation}
where 
$$\Phi_{2}(x,\eta,t,r)=\frac14 \sum_{|\a|=2}\partial^\alpha\Phi(x,\eta)\frac{(t,r)^\alpha}{\alpha!}$$
and
the remainder term $\Phi_{3}$ is exhibited in \eqref{eq:c11}.
Similarly, 
\begin{equation}\label{E2}
	\Phi\left(x-\frac{t}{2},\eta-\frac{r}{2}\right)=\Phi\phas-\frac {t}2\Phi_x \phas-\frac {r}2 \Phi_\eta \phas +{\Phi}_{2}(x,\eta,t,r)+\widetilde{\Phi}_{3}(x,\eta,t,r),
\end{equation}
with
$\widetilde{\Phi}_{3}$ defined in \eqref{eq:c12}.
Inserting in \eqref{KI} we obtain the espression of the Wigner kernel $k_W$ in \eqref{intro10}, where the symbol 
$\tilde{\sigma}$ is defined in \eqref{sigmatilde}.
Next, we work on the symbol $\tilde \sigma$. Observe that $\tilde\sigma\in\cC^\infty(\bR^{4d})$.
For $u=(x,\eta)$, $v=(t,r)$, by means of Leibniz's formula we can write
$$\partial^{\alpha'}_u \tilde\sigma(u,v)=
\sum_{\b+\gamma+\delta\leq \a'}
C_{\b,\gamma,\delta}
\partial_u^\b\left(e^{2\pi
	i[\Phi_{3}-\widetilde{\Phi}_{3}](u,v)}\right) \partial^\gamma_u \sigma(u+v/2)
\overline{\partial_u^\delta{\sigma}(u-v/2)}.$$
Now,
\begin{equation}\label{Expe0}
	\partial_u^\b\left(e^{2\pi
	i[\Phi_{3}-\widetilde{\Phi}_{3}](u,v)}\right) =e^{2\pi
	i[\Phi_{3}-\widetilde{\Phi}_{3}](u,v)}\sum_{|\a|=3,\,\tilde\beta\leq\beta} f_{\a,\tilde\beta}(\Phi )(u,v)(v^{\a})^{|\tilde\beta|}
\end{equation}
where $f_{\a,\tilde\beta}(\Phi )$ is a linear combination of products of functions
$$h_{\a,\theta}(\Phi )(u,v)= \int_0^1(1-\tau)\left[\partial^{\a+{\theta}}
\Phi(u+\frac{\tau}{2} v)-\partial^{\a+{\theta}}
\Phi(u-\frac{\tau}{2} v)\right]\,d\tau$$
with $\theta\leq\a'$ and
$$|h_{\a,\theta}(\Phi )(u,v)|\leq C_{\alpha,{\a'}},\,\quad\forall (u,v)\in\rdd,$$
by assumption \eqref{intro9}.
Hence we can estimate
$$\left|\partial_u^{\a'}\left(e^{2\pi
	i[\Phi_{3}-\widetilde{\Phi}_{3}](u,v)}\right) \right|\leq C_{\alpha'}\la v\ra^{3|\a'|},\quad\forall (u,v)\in\rdd.$$
Using Leibniz's formula for the partial derivatives with respect to $v$ we infer 
\begin{align*}\partial^{\beta'}_v \tilde\sigma(u,v)&=
\sum_{\b+\gamma+\delta\leq \beta'}C_{\b,\gamma,\delta}\,
(-1)^{|\delta|}2^{-(|\gamma|+|\delta|)}
\partial_v^\b\left(e^{2\pi
	i[\Phi_{3}-\widetilde{\Phi}_{3}](u,v)}\right) \partial^\gamma_v \sigma(u+v/2)
\overline{\partial_v^\delta{\sigma}(u-v/2)}\\
&=\sum_{\b+\gamma+\delta\leq \beta'}\widetilde C_{\b,\gamma,\delta}\,
\partial_v^\b\left(e^{2\pi
	i[\Phi_{3}-\widetilde{\Phi}_{3}](u,v)}\right) \partial^\gamma_v \sigma(u+v/2)
\overline{\partial_v^\delta{\sigma}(u-v/2)}.
\end{align*}
Computing similarly the derivatives of the complex exponential 
$$\partial_v^\b\left(e^{2\pi
	i[\Phi_{3}-\widetilde{\Phi}_{3}](u,v)}\right) $$
 we obtain the global estimate
$$ |\partial^{\a}_u\partial^{\b}_v\tilde\sigma(u,v)|\leq C_{\a,\beta}\la v\ra^{3(|\a|+|\beta|)},\quad \a,\beta\in\bN^{2d},$$
that is $\tilde\sigma\in E(\rdd\times\rdd)$.

In particular, if $\Phi$ is quadratic then $\tilde\Phi_3\equiv 0$ and the assumption $\sigma\in S^0_{0,0}(\rdd)$ yields $\tilde\sigma \in S^0_{0,0}(\bR^{4d})$, cf. \cite{CGR2024-1}.
\end{proof}

Next, we focus on the proof of Theorem \ref{intro-thm2}. We need two auxiliary propositions.

	\begin{proposition}\label{propL41}
		Let $m_j\in\cC(\rd)$, $j=1,2$, with
		\[
			|m_j(x)|\lesssim\la x\ra^{-s}, \qquad j=1,2, \quad x\in\rd, \quad s>d.
		\]
		Then, the convolution $m=m_1\ast m_2$ is in $\cC(\rd)$, with
		\[
			|m(x)|\lesssim\la x\ra^{-s}, \qquad x\in\rd.
		\]
	\end{proposition}
	The proof is available in several textbooks, see for example \cite[Lemma 11.1.1]{book}. 
	\begin{proposition}\label{propL42}
	Consider $\gamma\in\bN^d$ and the multiplier $v_M$ in \eqref{vM}, with $M\leq-|\gamma|-d-1$. For every $\alpha\in\bN^d$, we have $\partial_x^\gamma(x^\alpha v_M(x))\in\cC(\rd)$, and
		\[
			|\partial_x^\gamma(x^\alpha v_M(x))|\leq C\la x\ra^{-2P}, \qquad \forall P\in\bN,
		\]
		for a constant $C$ depending on $M,\gamma, d,\alpha,P$. In particular, if $M\leq -d-1$, then
		\begin{equation}\label{eqL41}
			|v_M(x)|\leq C\la x\ra^{-2N}, \qquad \forall N\in\bN,
		\end{equation}
	where $v_M$ is the multiplier defined in \eqref{vM}.
	\end{proposition}
	\begin{proof}
		To be definite, we recall the easy argument. Write
		\[
			x^\beta\partial_x^\gamma(x^\alpha v_M(x))=C_{\alpha,\beta,\gamma}\int_{\rd}e^{2\pi ix\omega}\partial_\omega^\beta(\omega^\gamma\partial^\alpha_\omega\la \omega\ra^M)d\omega.
		\]
		If $|\gamma|+M\leq -d-1$, then for every $\beta\in\bN^d$,
		\[
			|\partial^\beta_\omega(\omega^\gamma\partial^\alpha\la\omega\ra^M)|\lesssim \la \omega\ra^{-d-1}.
		\]
		By taking $|\beta|=2P$ the estimates follow. For $\gamma=\alpha=0$, we obtain \eqref{eqL41}.
	\end{proof}

	\begin{proposition} \label{propL43}	For $\Phi(x,\xi)=x\xi$ in \eqref{intro4}, consider the pseudo-differential operator  in the Kohn-Nirenberg form $T_I=\sigma(x,D)$ in \eqref{K-N}
		with $\sigma\in S^0_{0,0}(\rdd)$. Then, from Proposition \ref{prop-intro1}, we have the corresponding Wigner kernel
		\[
		k_W(x,\xi,y,\eta)=\int_{\rdd}e^{-2\pi i[t(\xi-\eta)+r(y-x)]}\tilde\sigma(x,\eta,t,r)dtdr,
		\]
		where $\tilde\sigma\in S^0_{0,0}(\bR^{2d}_{x,\eta}\times\bR^{2d}_{t,r})$.
		For $z=(x,\xi)$ and $w=(y,\eta)$ in $\rdd$,  we rephrase $k_M(z,w)$ in \eqref{intro14} as
		\begin{equation}
				\label{eqL43}
				k_{W,M}(z,w)=\la D_x\ra^M\la D_\eta\ra^M k_{W,M}^\#(z,w),
		\end{equation}
		with
		\begin{equation}\label{eqL44}
			k_{W,M}^\#(z,w)=\la D_\xi\ra^M\la D_y\ra^M k_W(z,w).
		\end{equation}
Then
	\begin{equation}\label{eqL45}
		|k_{W,M}^\#(z,w)|\lesssim\la z-w\ra^{-2S}, \qquad S\in\bN.
	\end{equation}
	\end{proposition}
	\begin{proof}
		By convolution with $v_M\otimes v_M$, with $v_M$ multiplier in \eqref{vM}, we have from \eqref{eqL44},
		\begin{align*}
			k_{W,M}^\#(z,w)&=\int_{\bR^{2d}}k_W(x,\xi-\tilde\xi,y-\tilde y,\eta)v_M(\tilde\xi)m_M(\tilde y)d\tilde\xi d\tilde y\\
			&=\int_{\bR^{4d}}e^{-2\pi i[t(\xi-\tilde\xi-\eta)+r(y-\tilde y-x)]}\tilde\sigma(x,\eta,t,r)v_M(\tilde\xi)v_M(\tilde y)d\tilde\xi d\tilde ydtdr.
		\end{align*}
		Writing
		\[
			\la 2\pi(\xi-\eta,x-y) \ra^{-2S}(1-\Delta_{t,r})^Se^{-2\pi i[t(\xi-\eta)+r(y-x)]}=e^{-2\pi i[t(\xi-\eta)+r(y-x)]},
		\]
		and integrating by parts we obtain:
		\[
		k^\#_{W,M}(z,w)=\frac{1}{\la 2\pi(z-w)\ra^{2S}}\int_{\bR^{4d}}e^{-2\pi i[t(\xi-\eta)+r(y-x)]}b(x,\eta,\tilde\xi,\tilde y,t,r)v_M(\tilde\xi)v_M(\tilde y)d\tilde\xi d\tilde ydtdr,
		\]
		where
		\begin{align*}
			b(x,\eta,\tilde\xi,\tilde y,t,r)&=(1-\Delta_{t,r})^S(e^{2\pi i(t\tilde\xi+r\tilde y)}\tilde\sigma(x,\eta,t,r))\\
			&=e^{2\pi i(t\tilde \xi+r\tilde y)}\sum_{|\alpha|+|\beta|\leq 2S}\sigma_{\alpha\beta}(x,\eta,t,r)\tilde\xi^\alpha\tilde y^\beta,
		\end{align*}
		where $\sigma_{\alpha\beta}\in S^0_{0,0}(\bR^{2d}_{x,\eta}\times\bR^{2d}_{t,r})$. Write now, with $L=d+1$,
		\begin{equation}\label{L}
			e^{2\pi i(t\tilde\xi+r\tilde y)}=\la 2\pi(t,r)\ra^{-2L}(1-\Delta_{\tilde\xi,\tilde y})^Le^{2\pi i(t\tilde \xi+r\tilde y)},
		\end{equation}
		and integrate again by parts. We obtain,
		\begin{equation}\label{eqL46}
			k_{W,M}^\#(z,w)=\frac{1}{\la 2\pi (z-w)\ra^{2S}}I(z,w),
		\end{equation}
		with
		\begin{equation}\label{eqL47}
			I(z,w)=\int_{\bR^{4d}}e^{-2\pi i[t(\xi-\tilde\xi-\eta)+r(y-\tilde y-x)]}q(x,\eta,\tilde\xi,\tilde y,t,r)d\tilde\xi d\tilde ydtdr,
		\end{equation}
		being:
		\begin{equation}\label{eqL48}
			q(x,\eta,\tilde\xi,\tilde y,t,r)=\la 2\pi(t,r)\ra^{-2d-2}\sum_{|\alpha|+|\beta|\leq 2N}\sigma_{\alpha\beta}(x,\eta,t,r)(1-\Delta_{\tilde\xi,\tilde y})^{d+1}(\tilde\xi^\alpha\tilde y^\beta v_M(\tilde\xi)v_M(\tilde y)).
		\end{equation}
		We can estimate:
		\begin{align*}
			|(1-\Delta_{\tilde\xi,\tilde y})^{d+1}(\tilde\xi^\alpha\tilde y^\beta v_M(\tilde \xi)v_M(\tilde y))|&\leq \sum_{|\gamma|+|\delta|\leq2d+2}|\partial^\gamma_{\tilde\xi}(\tilde\xi^\alpha v_M(\tilde\xi))| |\partial^\delta_{\tilde y}(\tilde y^\beta v_M(\tilde y))|\\
			&\lesssim \la \tilde\xi\ra^{-2P}\la\tilde y\ra^{-2P}, \qquad P\in\bN,
		\end{align*}
		since $|\gamma|\leq2d+2$, $|\delta|\leq 2d+2$, $-M=3d+3$, hence $-M\geq|\gamma|+d+1$,  and we can apply Proposition \ref{propL42}. Choosing $2P\geq d+1$, we have that $I(z,w)$ in \eqref{eqL47}, \eqref{eqL48} is a continuous function with
		\begin{align*}
			|I(z,w)|&\leq\int_{\bR^{4d}}|q(x,\eta,\tilde\xi,\tilde y,t,r)|d\tilde\xi d\tilde ydtdr\\
			&\lesssim \int_{\bR^{4d}}\la(t,r)\ra^{-2d-2}\la\tilde\xi\ra^{-d-1}\la\tilde y\ra^{-d-1}d\tilde\xi d\tilde ydtdr<\infty.
		\end{align*}
		Hence, from \eqref{eqL46} we deduce \eqref{eqL45} and we are done.
	\end{proof}
	\begin{proof}[Proof of Theorem \ref{intro-thm2}]\label{sec:thm1proof}
			Consider $\Phi(x,\xi)=x\xi$ in \eqref{intro4}, and the related pseudo-differential operator $T_I=\sigma(x,D)$ in the Kohn-Nirenberg form \eqref{K-N}.
 Let us prove \eqref{intro17} with
		\begin{equation}\label{eqL42}
			M=-3d-3.
		\end{equation}
	
	From \eqref{eqL43}, we have:
	\[
		|k_{W,M}(z,w)|\leq\int_{\rdd}v_M(x-\tilde x)v_M(\eta-\tilde\eta)|k_{W,M}^\#(\tilde x,y,\xi,\tilde\eta)|d\tilde xd\tilde\eta.
	\]
	Estimating $k_{W,M}^\#$ as in Proposition \ref{propL43} with $T=2N$, we have:
	\begin{align*}
		|k_{W,M}^\#(\tilde x,y,\xi,\tilde\eta)|&\lesssim \la(\tilde x-y,\tilde\eta-\xi)\ra^{-4N}\lesssim\la\tilde x-y\ra^{-2N}\la \tilde\eta-\xi\ra^{-2N}.
	\end{align*}
	Since $M\leq -d-1$, in view of our assumption \eqref{eqL42}, we may apply \eqref{eqL41} in Proposition \ref{propL42} to $v_M$. Hence, for $2N>d$,
	\begin{align*}
		|k_{W,M}(z,w)|&\lesssim \int_{\rd}\la x-\tilde x\ra^{-2N}\la \tilde x-y\ra^{-2N}d\tilde x\int_{\rd}\la\eta-\tilde\eta\ra^{-2N}\la\tilde\eta-\xi\ra^{-2N}d\tilde\eta\\
		&\lesssim \la x-y\ra^{-2N}\la\xi-\eta\ra^{-2N}\\
		&\lesssim \la z-w\ra^{-2N},
	\end{align*}
	where we applied Proposition \ref{propL41} and the translation invariance of the convolution. This concludes the proof.
\end{proof}

To prove Theorem \ref{intro-thm3} we need a few preliminaries.
The counterpart of Proposition \ref{propL43} is now the following
\begin{proposition}\label{Prop5.1}
Assume $\Phi$ and $\sigma$ as in the hypotheses of Proposition \ref{prop-intro1}.  Set \begin{equation}\label{5.2}
	\f(t)=e^{-2\pi i t^2},\quad t\in\rd.
\end{equation}
Define
\begin{equation}
	\label{eq5.3}
	k_{W,G}(z,w)=(\f\otimes \f)\ast_{x,\eta}  k_{W,G}^\#(x,\xi,y,\eta),
\end{equation}
with convolution with respect to the variables $x,\eta$, where
\begin{equation}\label{eq5.4}
	k_{W,G}^\#(z,w)=(\f\otimes \f)\ast_{\xi,y} k_W(x,\xi,y,\eta),
\end{equation}
with convolution with respect to $\xi,y$.  Introduce 
\begin{equation}\label{intro16}
\lambda_\Phi(z,w):=\la( \xi-\Phi_x(x,\eta),y-\Phi_\eta(x,\eta))\ra
=(1+|\xi-\Phi_x(x,\eta)|^2+|y-\Phi_\eta(x,\eta)|^2)^{1/2}.
\end{equation}
Then,
$$|	k_{W,G}^\#(z,w)|\lesssim \lambda_\Phi(z,w)^{-2S},\qquad\forall S\in \bN.$$
\end{proposition}
\begin{proof} We use the Wigner kernel $k_W$ in formula \eqref{intro10} of Proposition \ref{prop-intro1}. Writing for short
\begin{equation}\label{5.6}
\rho_\Phi:=\rho_\Phi(x,y,\xi,\eta,t,r)=t(\xi-\Phi_x(x,\eta))+r(y-\Phi_\eta(x,\eta))
\end{equation}
we have
\begin{equation}\label{5.7}
	k_{W,G}^\#(x,\xi,y,\eta)=\intrdd e^{-2\pi i \rho_\Phi}e^{-2\pi i (t\tilde\xi+r\tilde y)}\tilde\sigma(x,\eta,t,r)\f(\tilde\xi)\f(\tilde y)d\tilde\xi d\tilde ydt dr,
\end{equation}
with $\tilde\sigma\in E(\rdd\times\rdd)$, cf. Definition \ref{defe1}. Writing 
\begin{equation}\label{5.8}
	\lambda_\Phi(z,w)^{-2S}(1-\Delta_{t,r})^S e^{-2\pi i \rho_\Phi}= e^{-2\pi i \rho_\Phi}
\end{equation}
and integrating by parts we obtain
\begin{equation}\label{5.9}
	k_{W,G}^\#(x,\xi,y,\eta)=\lambda_\Phi(z,w)^{-2S}\intrdd e^{-2\pi i \rho_\Phi}b(x,\eta,\tilde \xi,\tilde y,t,r)\f(\tilde\xi)\f(\tilde y)d\tilde\xi d\tilde ydt dr,
\end{equation}
where, similarly to the proof of Proposition \ref{propL43}, we write
\begin{align}\label{5.10}
b(x,\eta,\tilde\xi,\tilde y,t,r)&=(1-\Delta_{t,r})^S(e^{2\pi i (t\tilde\xi+r\tilde y)}\tilde\sigma(x,\eta,t,r))\\
&=e^{2\pi i (t\tilde\xi+r\tilde y)}\sum_{|\a|+|\b|\leq 2S}\tau_{\a,\beta}(x,\eta,t,r)\tilde\xi^\a\tilde y^\beta.
\end{align}
Since 
$$|D^\a_tD^\beta_r\tilde\sigma(x,\eta,t,r)|\lesssim \la (t,r)\ra^{3(|\a|+|\beta|)},\quad (\a,\beta)\in \bN^{2d},$$
we estimate, for $|\a|+|\beta|\leq 2S$,
\begin{equation}\label{5.11}
|\tau_{\a,\beta}(x,\eta,t,r)|\lesssim \la (t,r)\ra^{6S}.  
\end{equation}
Using formula \eqref{L}  for $L=3S+d+1$, and integrating by parts in \eqref{5.9} we obtain
\begin{equation}\label{5.12}
	k_{W,G}^\#(z,w)=\lambda_\Phi(z,w)^{-2S}J(z,w)
\end{equation}
where 
\begin{equation}\label{5.13}
	J(z,w)=\intrdd e^{-2\pi i \rho_\Phi}e^{-2\pi i (t\tilde\xi+r\tilde y)}p(x,\eta,\tilde\xi,\tilde y,t,r)d\tilde\xi d\tilde ydt dr,
\end{equation}
with
\begin{align*}
	p&(x,\eta,\tilde\xi,\tilde y,t,r)=\\
	&\quad \la (t,r)\ra^{-6S-2d-2}\sum_{|\a|+|\b|\leq 2s}\tau_{\a,\beta}(x,\eta,t,r)(1-\Delta_{\tilde\xi,\tilde y})^{3S+d+1}(\tilde\xi^\a\tilde y^\beta\f(\tilde\xi)\f(\tilde y) )d\tilde\xi d\tilde ydt dr.
\end{align*}
Since, in view of \eqref{5.11},
$$\la (t,r)\ra^{-6S-2d-2}| \tau_{\a,\beta}(x,\eta,t,r)|\lesssim \la (t,r)\ra^{-2d-2} $$
and 
$$(1-\Delta_{\tilde\xi,\tilde y})^{3S+d+1}(\tilde\xi^\a\tilde y^\beta\f(\tilde\xi)\f(\tilde y))\in\cS(\rdd_{\tilde\xi,\tilde y}),
$$
we conclude from \eqref{5.13}
$$|J(z,w)|<\infty,$$
and the thesis follows from \eqref{5.12}.
\end{proof}

\begin{proposition}\label{Prop5.2}
If either the phase $\Phi$ is tame (c.f. Definition \ref{defe2}) or of the form \eqref{introbis}, the system in \eqref{intro6}
can be solved with respect to $z=(x,\xi)$ by $\chi:\rdd\to\rdd$
defined by 
\begin{equation}\label{chi}
	(x,\xi)=\chi(y,\o),
\end{equation}
hence $z=\chi(w)$, $w=(y,\eta)$, with $\chi\in\cC^\infty(\rdd)$ and 
\begin{equation}\label{5.15}
	\la z-\chi(w)\ra\asymp \lambda_\Phi(z,w),
\end{equation}
where $\lambda_\Phi$ is defined in \eqref{intro16}.
\end{proposition}
\begin{proof}
	If $\Phi$ is tame we refer to \cite[Lemma 6.1.4 and Lemma 6.4.2]{Elena-book}. In the case of $\Phi$ of the form \eqref{introbis} the proof is obvious.
\end{proof}
\begin{proposition}\label{Prop5.3}
Under the assumptions of Proposition \ref{Prop5.2} the system \eqref{intro6} can be solved with respect to $(x,\eta)$ by $x=\psi_1(y,\xi)$, $\eta=\psi_2(y,\xi)$, with
\begin{equation}\label{5.16}
		\la z-\chi(w)\ra\asymp \la (x-\psi_1(y,\xi),\eta-\psi_2(y,\xi))\ra.
	\end{equation}
\end{proposition}
\begin{proof}
Let $\Phi$ be a tame phase function. Observe that 
$\eta=\psi_2(y,\chi_2(y,\eta))$ by definition of $\chi$ and $\psi=(\psi_1,\psi_2)$. Hence
$$|\eta-\psi_2(y,\xi)|=|\psi_2(y,\chi_2(y,\eta))-\psi_2(y,\xi)|\leq C|\xi-\chi_2(y,\eta)|$$
by property \eqref{tame0} of tame phases. Moreover,
$$|x-\psi_1(y,\xi)|\leq |x-\chi_1(y,\xi)|+|\chi_1(y,\eta)-\psi_1(y,\xi)|.$$
Since $\chi_1(y,\eta)=x=\psi_1(y,\eta)$ and $\psi_1(y,\xi)=\psi_1(y,\chi_2(y,\eta))$
we can estimate
$$|x-\psi_1(y,\xi)|\leq|x-\chi_1(y,\xi)|+C|\xi-\chi_2(y,\eta)|.$$
Hence, we obtain
 $$\la( x-\psi_1(y,\xi),\eta-\psi_2(y,\xi))\ra\lesssim \la z-\chi(w)\ra.$$
 The converse estimate is proved similarly. \par The case $\Phi$ of the form \eqref{introbis} is obvious.
\end{proof}

The preceding propositions are the pillars for the proof of Theorem \ref{intro-thm3}.
\begin{proof}[Proof of Theorem \ref{intro-thm3}]
	We argue as in the proof of Theorem \ref{intro-thm2}, by replacing Bessel smoothing with the Gaussian one. We split $G(z,w)=e^{-2\pi(z^2+w^2)}$ defined in \eqref{intro12} as
	\begin{equation}\label{5.1}
		G(z,w)=\f(x)\f(\xi)\f(y)\f(\eta),\quad z=(x,\xi),\,\,w=(y,\eta),
	\end{equation}
	where $\f$ is   the Gaussian in \eqref{5.2}.	Then by Proposition \ref{Prop5.1} 
	\begin{equation*}
		|k_{W,G}^\#(z,w)|\lesssim \lambda_\Phi^{-2S}, \qquad \forall S\in\bN.
	\end{equation*}
From \eqref{eq5.3} we obtain 
\begin{equation}
	\label{5.17}
	|k_{W,G}(z,w)|=\intrdd\f(x-\tilde x) \f(\eta-\tilde\eta)|  k_{W,G}^\#(\tilde x,\xi,y,\tilde \eta)|d\tilde xd\tilde\eta.
\end{equation}
We estimate $k_{W,G}^\#$ as in Proposition \eqref{Prop5.1} by using Proposition \ref{Prop5.2} and Proposition \ref{Prop5.3}:
\begin{align*}
|k_{W,G}^\#(z,w)|&\lesssim \lambda_\Phi(z,w)^{-2S}\lesssim \la z-\chi(w)\ra^{-2S}\lesssim \la( x-\psi_1(y,\xi),\eta-\psi_2(y,\xi))\ra^{-2S}.
\end{align*}
Taking $S=N$, we estimate
$$|k_{W,G}^\#(\tilde x, y,\xi,\tilde\eta)|\lesssim \la(\tilde x -\psi_1(y,\xi),\tilde\eta-\psi_2(y,\xi))\ra^{-2N},$$
hence, by \eqref{5.17},
\begin{align*}
|k_{W,G}^\#(z,w)|\lesssim \left(\intrd \f(x-\tilde x)\la \tilde x-\psi_1(y,\xi)\ra^{-2N}d\tilde x\right)\left(\intrd \f(\eta-\tilde \eta)\la \tilde \eta-\psi_2(y,\xi)\ra^{-2N}d\tilde \eta\right),
\end{align*}
where we may further estimate 
$$\f(x-\tilde x)\lesssim\la x-\tilde x\ra^{-2N},\quad \f(\eta-\tilde \eta)\lesssim\la \eta-\tilde \eta\ra^{-2N}.$$
From the translation invariance of the convolution and the weight property in Proposition \ref{propL41} with $2N>d$, we conclude
\begin{align*}
|k_{W,G}^\#(z,w)|&\lesssim \la x-\psi_1(y,\xi)\ra^{-2N}\la \eta-\psi_2(y,\xi)\ra^{-2N}\\
&\lesssim \la  (x-\psi_1(y,\xi),  \eta-\psi_2(y,\xi))\ra^{-2N}\\
&\lesssim \la z-\chi(w)\ra^{-2N},
\end{align*}
where we used Proposition \ref{Prop5.3} again. This concludes the proof.
\end{proof}

We now turn our attention to the Gabor matrix of the operator $T$, showing that results of the Wigner kernel allow to infer decay estimates for the Gabor matrix, as exhibited in Corollary \ref{intro-cor4}.
\begin{proof}[Proof of Corollary \ref{intro-cor4}]\label{sec:thm3proof}
As observed in the Introduction, this proof follows from \eqref{intro25}. For a detailed proof of \eqref{intro25} we address to \cite{CGR2024-2}. For the benefit of the reader, we sketch here the argument. Namely, from the covariance of the Wigner distribution we have
	\begin{align*}
	|\la T\pi(z)g,\pi(w)\gamma\ra|^2&=\la T\pi(z)g,\pi(w)\gamma\ra\overline{\la T\pi(z)g,\pi(w)\gamma\ra}\notag\\
	&=\la T\pi(z)g\otimes\overline{T\pi(z)g},\pi(w)\gamma\otimes\overline{\pi(w)\gamma}\ra\notag\\
	&=\la W(T\pi(z)g),W(\pi(w)\gamma)\ra\notag\\
	&=\la k_W, W(\pi(w)\gamma)\otimes W(\pi(z)g)\ra\notag\\
	&=\int_{\bR^{4d}}k_W(u,v)W(\pi(w)\gamma)(u)W(\pi(z)g)(v)dudv\notag\\
	&=\int_{\bR^{4d}}k_W(u,v)W\gamma(u-w)Wg(v-z)dudv\label{E2}\\
	&=[k_W\ast (W\gamma\otimes Wg)](w,z),\quad z,w\in\rdd.\notag
\end{align*}
Taking $g=\gamma$ as in \eqref{intro24} we obtain in particular 
$$(W\gamma \otimes Wg)(z,w)=e^{-2\pi(z^2+w^2)}=G(z,w),$$
according to the notation in \eqref{intro12}, see for example \cite[Lemma 1.3.12]{book}. Thus, formula \eqref{intro25} and Corollary \ref{intro-cor4} follow.
\end{proof}

\section{Examples}\label{sec:examples}
Elementary applications of the results in the preceding sections are provided for the propagators of Schr\"odinger equation
\begin{equation}\label{sec7-1}
	\begin{cases}
		i\partial_tu-2\pi \f(D)u=0 & \text{$x\in\rd$, $t\in\bR$},\\
		u(0,x)=u_0(x),
	\end{cases}
\end{equation}
	where $\f(D)$ is the Fourier multiplier 
\begin{equation}\label{sec7-2}
	\f(D)f(x)=\int_{\rd}e^{2\pi i\xi x}\f(\xi)\hat f(\xi)d\xi.
\end{equation}
By Fourier transforming, a solution of \eqref{sec7-1} is given by a FIO of type I:
\begin{equation}\label{sec7-3}
	u(t,x)=\int_{\rd}e^{2\pi i\Phi_t(x,\xi)}\hat u_0(\xi)d\xi,
\end{equation}
with phase function
\begin{equation}\label{sec7-4}
	\Phi_t(x,\xi)=x\xi-t\f(\xi).
\end{equation}
Under the assumption
\begin{equation}\label{sec7-5}
	|\partial^\alpha\f(\xi)|\leq C_\alpha \qquad \mbox{for} \,\,|\alpha|\geq3, \quad \xi\in\rd,
\end{equation}
we may apply Proposition \ref{prop-intro1} and obtain the Wigner kernel $k_{W,t}$ of the operator \eqref{sec7-3}, namely
\begin{equation}\label{sec7-6}
	k_{W,t}(x,\xi,y,\eta)=\int_{\rdd}e^{-2\pi i[s(\xi-\eta)+r(y-x+t\f_\eta(\eta))]}a_t(\eta,r)dsdr,
\end{equation}
where $a_t\in E(\rdd_{x,\eta}\times\rdd_{s,r})$ is in the present case independent of $x,s$. Considering the Hamiltonian flow
\begin{equation}\label{sec7-7}
	z=\chi_t(w)=(y+t\f_\eta(\eta),\eta), \qquad z=(x,\xi), \ w=(y,\eta),
\end{equation}
from Theorem \ref{intro-thm3} we have:
\begin{equation}\label{sec7-8}
	| k_{W,G,t}(z,w)|\lesssim\frac{1}{\la z-\chi_t(w)\ra^{2N}}, \qquad N\in\bN,
\end{equation}
where $k_{W,G,t}$ is the Gaussian smoothing of $k_{W,t}$, as in \eqref{intro11}.\par

Aiming to give an explicit example of \eqref{sec7-6}, \eqref{sec7-7} and \eqref{sec7-8}, we consider the elementary case of the cubic Hamiltonian in dimension $d=1$
\begin{equation}\label{9}
	\f(\xi)=\xi^3,\quad\mbox{i.e.}, \quad \f(D)=D^3_x,
\end{equation}
where we mean $D_x=\frac{1}{2\pi i}\partial_x$. Fixing the time $t=1$ for simplicity,  from \eqref{sec7-3}  we have the operator 
\begin{equation}\label{10}
	u(1,x)=T_If(x)=\int_{\bR} e^{2\pi i \Phi\phas}\hat{f}(\xi) \,d\xi
\end{equation} 
with 
\begin{equation}\label{11}
\Phi\phas=x\xi-\xi^3.
\end{equation}
The preceding formula \eqref{sec7-6} for the Wigner kernel reads
\begin{equation}\label{12}
	k_W(x,\xi,y,\eta)=\int_{\bR^2} e^{-2\pi i[s(\xi-\eta)+r(y-x+3\eta^2)]}a(r)drds
\end{equation}
where to compute the symbol $a(r)$, depending only on $r\in\bR$, we use Theorem \ref{teor3.6} above.
We have
\begin{equation}\label{13}
	k_W(x,\xi,y,\eta)=\int_{\bR^2} e^{2\pi i \tau(x,\xi,y,\eta,s,r)} dsdr,
\end{equation}
with 
\begin{align}\notag
	\tau(x,\xi,y,\eta)&=\Phi\left(x+\frac{s}{2},\eta+\frac{r}{2}\right)-\Phi\left(x-\frac{s}{2},\eta-\frac{r}{2}\right)-\left(\xi s+y r\right)\\ \label{14}
	&=\left(x+\frac{s}{2}\right)\left(\eta+\frac{r}{2}\right)-\left(\eta+\frac{r}{2}\right)^3\!\!\!-\left(x-\frac{s}{2}\right)\left(\eta-\frac{r}{2}\right)+\left(\eta-\frac{r}{2}\right)^3\!\!\!-\left(\xi s+yr\right)\\ \notag
	&=s(\eta-\xi)+r(x-y-3\eta^2)-\frac{r^3}{4}.\notag
\end{align}
Comparing \eqref{12} with \eqref{13} and \eqref{14}, we obtain that the symbol $a(r)$ in \eqref{12} is given by 
\begin{equation}\label{15}
	a(r)=e^{\pi i \frac{ r^3}{2}}.
\end{equation}
If we normalize the definition of the Airy function as
\begin{equation}\label{16}
	\Airy(\lambda)=\cF\left(e^{\pi i \frac{ r^3}{2}}\right)(\lambda),\quad r,\lambda \in\bR,
\end{equation}
where $\cF$ is the Fourier transform in $\cS'(\bR)$, we obtain from \eqref{12} the new expression for the Wigner kernel 
\begin{equation}\label{17}
k_W(x,\xi,y,\eta)=\Airy(y-x+3\eta^2)\delta_{\xi-\eta}.
\end{equation}
Starting from \eqref{17} we may simplify the smoothing arguments in the Proof of Theorem \ref{intro-thm3} as follows. Write
\begin{equation}\label{18}
	k_{W,G}^\#(x,\xi,y,\eta)=k_W(x,\cdot,\cdot,\eta)\ast_{y,\xi} e^{-2\pi (y^2+\xi^2)}=\psi(y-x+3\eta^2)e^{-2\pi (\xi-\eta)^2} 
\end{equation}
where
\begin{equation}\label{19}
	\psi(\lambda)=\Airy(\lambda)\ast e^{-2\pi \lambda^2}.
\end{equation}
The crucial point is to observe that $\psi\in\cS(\bR)$, see Remark \ref{rem1} below. Writing
\begin{equation}\label{20}
	z=\chi(w)=(y+3\eta^2,\eta),\quad z=(x,\xi),\,w=(y,\eta),
\end{equation}
from \eqref{18} we obtain
\begin{equation}\label{21}
|k_{W,G}^\#(z,w)|\lesssim \frac{1}{\la z-\chi(w)\ra^N},
\end{equation}
for every $N\in\bN$.
Arguing as in the proof of Theorem \ref{intro-thm3} we may obtain the same estimate for
$$k_{W,G}=k_{W,G}^\#\ast e^{-2\pi (x^2+\eta^2)}.$$
\begin{remark}\label{rem1} It is clear from \eqref{15} that we may expect as elements of the class $E(\bR^2\times\bR^2)$ in \eqref{intro8} functions $a(x,\eta,r,s)=a(r)$ of the sole variable $r\in\bR$ given by
	$$a(r)=e^{ic r^n},\,n\in\bN,\,n\geq 2,\,c\in\bR,\,c\not=0. $$
	Since $a\in \cC^{\infty}(\bR)$ and for every $\a\in\bN$
	$$|D^\alpha a(r)|\lesssim\l r\ra^{(n-1)\alpha},$$
	we have that $a(r)$ is a multiplier of $\cS(\bR)$, i.e., if $f\in\cS(\bR)$, then $a f\in\cS(\bR)$. Consequently, by applying the Fourier transform in $\cS(\bR)$, 
	$$\hat{a}(\lambda)=\cF(e^{ic r^n})(\lambda),\quad r,\lambda\in\bR,$$
	is a convolutor of $\cS(\bR)$, i.e., if $f\in\cS(\bR)$, then $\hat{a}\ast f\in\cS(\bR)$.
	According to the terminology of L. Schwartz \cite[Chapter VII, Sections 5 and 8]{Schwartz}, $a\in \mathcal{O}_M$, class of the slowly increasing functions, and $\hat{a}\in \mathcal{O}_C'$, class of the rapidly decreasing distributions.
\end{remark}
\begin{remark}\label{rem2}
	Several papers in the literature are devoted to results for the Wigner transform of oscillatory functions:
	$$W(e^{2\pi i t \f(x)}A(x))$$
	under different assumptions on the amplitude $A(x)$ and on the phase $\f(x)$, with semi-classical parameter $\eps=1/t$, see the classical paper of M. V. Berry \cite{Berry77} and subsequent contributions, cf.  for example B. Hanin, S. Zelditch \cite{HaninZelditch2022}. Regarding the oscillatory function as a multiplier or Fourier multiplier, we can meet our setting and deduce more precise results for $a_t$ in \eqref{sec7-6} and its transform. In the perspective of the estimates \eqref{sec7-8}, however, information on $a_t$ is cancelled by the exponential smoothing.
\end{remark}
Finally, consider the case of the free particle in dimension $d=1$, i.e.
$$\f(\xi)=\xi^2,\quad\f(D)=D^2_x,$$
in \eqref{sec7-3}, \eqref{sec7-4}, where again we fix $t=1$ for simplicity. The Wigner kernel in \eqref{sec7-6} reads
$$k_W(x,\xi,y,\eta)=\int_{\bR^2} e^{-2\pi i [s(\xi-\eta)+r(y-x+2\eta)]}\,dsdr=\delta_{y-x+2\eta,\xi-\eta}.
$$
Similarly to the pseudo-differential case in Theorem \ref{intro-thm2}, a smoothing by Bessel potentials is sufficient to obtain the estimates \eqref{sec7-8} for every $N\in\bN$. Namely,
$$k_{W,-2}^\#=\la D_y\ra^{-2}\la D_\xi\ra^{-2}k_W=(v_{-2}\otimes v_{-2})\ast_{y,\xi} k_W=v_{-2}(y-x+2\eta)v_{-2}(\xi-\eta).$$
Since $v_{-2}$ is a rapidly decreasing continuous function, explicitly
$$v_{-2}(x)=\int_{\bR} e^{2\pi i x\omega}\frac{1}{1+\omega^2} d\omega=\pi e^{-2\pi |x|},\quad x\in\bR,$$
the estimates
$$|k_{W,-2}(x,\xi,y,\eta)|=|\la D_x\ra^{-2}\la D_\eta\ra^{-2}k_{W,-2}^\#(x,\xi,y,\eta)|\lesssim \frac{1}{\la (x-(y+2\eta),\xi-\eta)\ra^N}$$
follow for every $N\in\bN$.

Note that for the cubic phase function \eqref{11} one could not replace the Gaussian by a Bessel $v_M$ smoothing, since in \eqref{19} the convolution with the Airy function would not be any more of rapid decay, see \emph{``Remarque importante"} at the end of page 244 in \cite{Schwartz}.
\section*{Acknowledgements}
The authors have been supported by the Gruppo Nazionale per l'Analisi Matematica, la Probabilità e le loro Applicazioni (GNAMPA) of the Istituto Nazionale di Alta Matematica (INdAM).
\begin{appendix}\label{AppA}
	\section{Ghostbusters}
	To validate the importance of our approach, we provide examples of how Gaussian convolution reduces ghost frequencies. To facilitate the graphical feedback, in the following discussion we will take $x,\xi\in\bR$. Let us consider the tempered distribution $1+\delta$, where $\la\delta,f\ra=\overline{f(0)}$ is the point mass measure. For the benefit of the reader, we report the computation of $W(1+\delta)$. For a given $g\in\cS(\rdd)$,
	\begin{align*}
		\la W\delta,g\ra&=\la \cF_2\mathfrak{T}_w(\delta\otimes\delta),g\ra=\la \delta\otimes\delta,\mathfrak{T}_w^{-1}\cF_2^{-1}g\ra=\overline{\mathfrak{T}_w^{-1}\cF_2^{-1}g(0,0)}=\overline{\cF_2^{-1}g(0,0)}\\
		&=\int_{\rd}\overline{g(0,y)}dy=\la \delta\otimes 1,g\ra.
	\end{align*}
	Consequently, $W\delta=\delta\otimes 1$. Analogously,
	\begin{align*}
		\la W1,g\ra&=\la 1\otimes1,\mathfrak{T}_w^{-1}\cF_2^{-1}g\ra=\la\delta\otimes\delta,\cF\mathfrak{T}_w^{-1}\cF_2^{-1}g\ra=\overline{\cF(\mathfrak{T}_w^{-1}\cF_2^{-1}g)(0,0)}\\
		&=\overline{\cF\cF_2^{-1}g(0,0)}=\overline{\cF_1g(0,0)}=\int_{\rd} \overline{g(z,0)}dz=\la 1\otimes\delta,g\ra,
	\end{align*}
	where $\cF_1F(\xi,y)=\int_{\rd}F(t,y)e^{-2\pi i\xi t}dt$ is the Fourier transform with respect to the time variables. 
	\begin{align*}
		\la W(1,\delta),g\ra&=\la 1\otimes\delta,\mathfrak{T}_w^{-1}\cF_2^{-1}g\ra=\int_{\rd}\overline{\cF_2^{-1}g(x/2,x)}dx\\
		&=\int_{\rd}\int_{\rd}\overline{g(x/2,y)}e^{-2\pi iyx}dydx=2^d\int_{\rdd}\overline{g(x',y)}e^{-4\pi iyx'}dx'dy\\
		&=\la 2^de^{-4\pi i x\xi},g\ra.
	\end{align*}
	Therefore, $W(1,\delta)(x,\xi)=2^de^{-4\pi ix\xi}$. Using the polarization property of the Wigner distribution \eqref{polarization}, we can compute $W(1+\delta)$:
	\begin{align*}	
		W(1+\delta)=W1+W\delta+2\Re(W(1,\delta))=1\otimes\delta+\delta\otimes 1+2^{d+1}\cos(4\pi x\xi).
	\end{align*}
	The 2D and 3D representations of $W(1+\delta)$ are outlined in figure \ref{fig:1}, where phantom frequencies stand out clearly.

	\begin{figure}
		\includegraphics[width=0.5\textwidth]{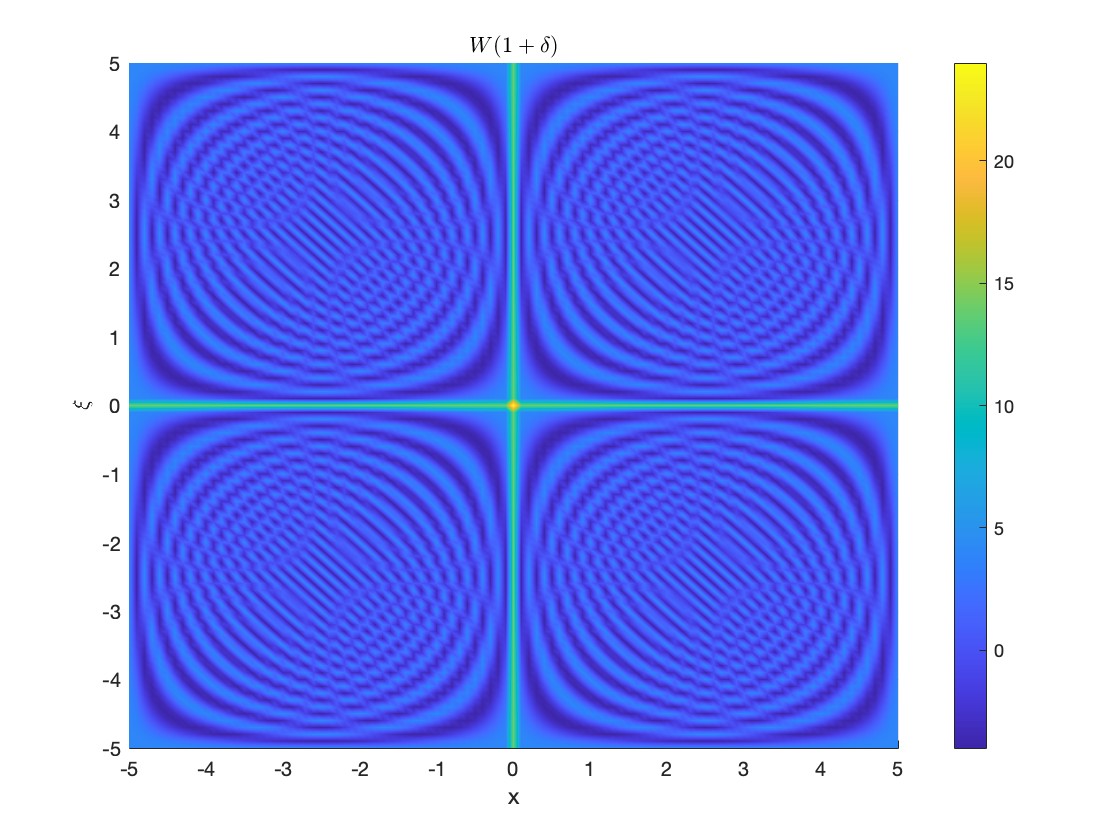}\includegraphics[width=0.5\textwidth]{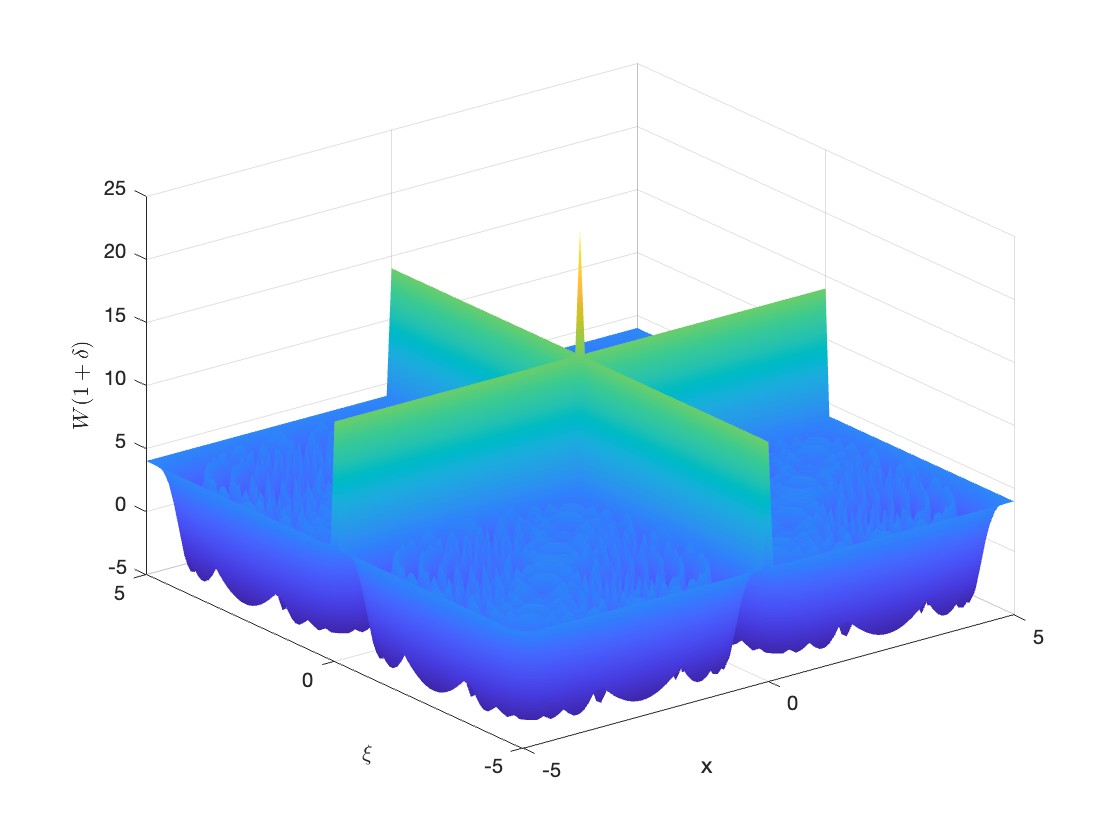}
		\caption{2D and 3D representations of $W(1+\delta)$.}
		\label{fig:1}
	\end{figure}

	Convolution with Gaussians reduce ghost frequencies, as it can be made evident by displaying the Husimi distribution of $1+\delta$, i.e., the convolution of $W(1+\delta)$ with a Gaussian kernel. To compute it, we use \eqref{husimi} as follows.
	\begin{align*}
		H(1+\delta)(x,\xi)&=|V_\f (1+\delta)(x,\xi)|^2=|V_\f 1(x,\xi)+V_\f\delta(x,\xi)|^2, \qquad x,\xi\in\rd.
	\end{align*}
	First, we compute $V_\f 1$ and $V_\f\delta$.
	\begin{align*}	
		V_\f1(x,\xi)&=\int_{\rd}\f(t-x)e^{-2\pi i\xi t}dt=e^{-2\pi ix\xi}\hat \f(\xi)=e^{-2\pi ix\xi} e^{-\pi |\xi|^2},
	\end{align*}
	$x,\xi\in\rd$. Using \eqref{fundidTF} and the previous identity,
	\begin{align*}
		V_\f\delta(x,\xi)=e^{-2\pi i\xi x}V_{\f}1(\xi,-x)=e^{-\pi |x|^2}, \qquad x,\xi\in\rd.
	\end{align*}
	Consequently,
	\begin{align*}
		H(1+\delta)(x,\xi)&=|e^{-\pi|x|^2}+e^{-2\pi i\xi x}e^{-\pi|\xi|^2}|^2\\
		&=(e^{-\pi|x|^2}+e^{-2\pi i\xi x}e^{-\pi|\xi|^2})(e^{-\pi|x|^2}+e^{2\pi i\xi x}e^{-\pi|\xi|^2})\\
		&=e^{-2\pi|x|^2}+e^{-\pi |\xi|^2}+2\cos(2\pi \xi x)e^{-\pi(|x|^2+|\xi|^2)}.
	\end{align*}
		\begin{figure}
		\includegraphics[width=0.5\textwidth]{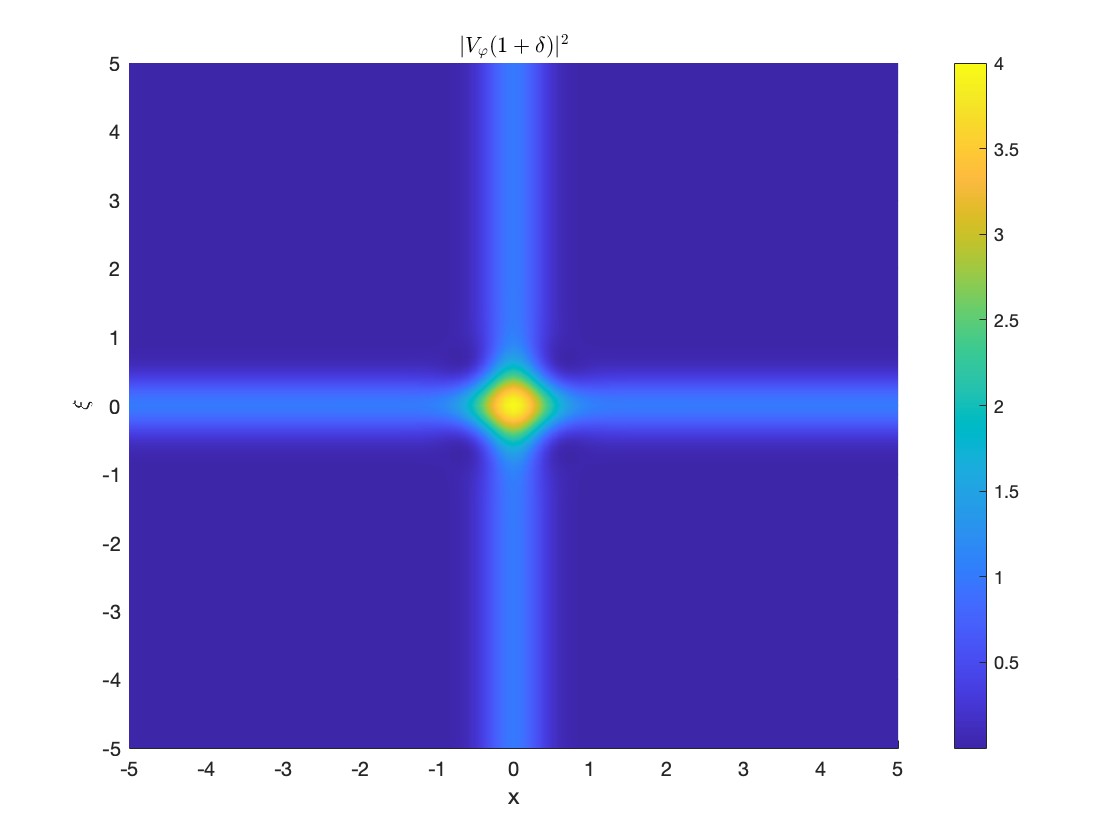}\includegraphics[width=0.5\textwidth]{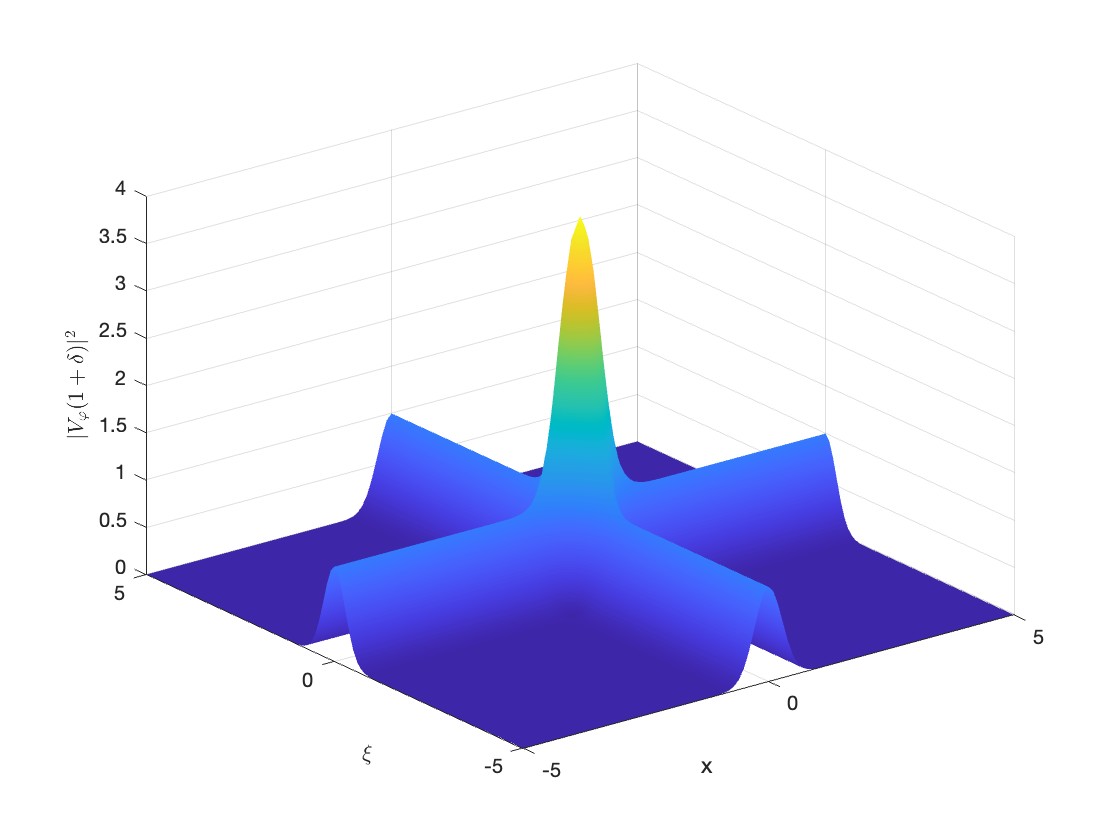}
		\caption{2D and 3D representations of $H(1+\delta)=|V_\f(1+\delta)|^2$.}
		\label{fig:2}
	\end{figure}
	Figure \ref{fig:2} displays the Husimi distribution of $1+\delta$, where the reduction of ghost frequencies, due to the convolution with the Gaussian, appears evident. The same nice abatement of ghost-frequencies cannot be achieved with Bessel-Sobolev potentials, as we discuss hereafter. 
	
	For, observe that:
	\begin{align*}
		(\delta\otimes 1)\ast e^{-|u|-|v|}(x,\xi)&=\int_{\rdd}e^{-|x|-|\xi-\eta|}d\eta=2\pi e^{-|x|}
	\end{align*}
	and, analogously, 
	\[
		(1\otimes\delta) \ast e^{-|u|-|v|}(x,\xi)=2\pi e^{-|\xi|}.
	\]
	Consequently,
	\begin{align*}
		W(1+\delta)\ast e^{-|u|-|v|}(x,\xi)&=[1\otimes\delta+\delta\otimes1+2^{d+1}\cos(2\pi uv)]\ast e^{-|u|-|v|}(x,\xi)\\
		&=2\pi(e^{-|x|}+e^{-|\xi|})+2^{d+1}\cos(2\pi uv)\ast e^{-|u|-|v|}(x,\xi).
	\end{align*}

	\begin{figure}
		\includegraphics[width=0.5\textwidth]{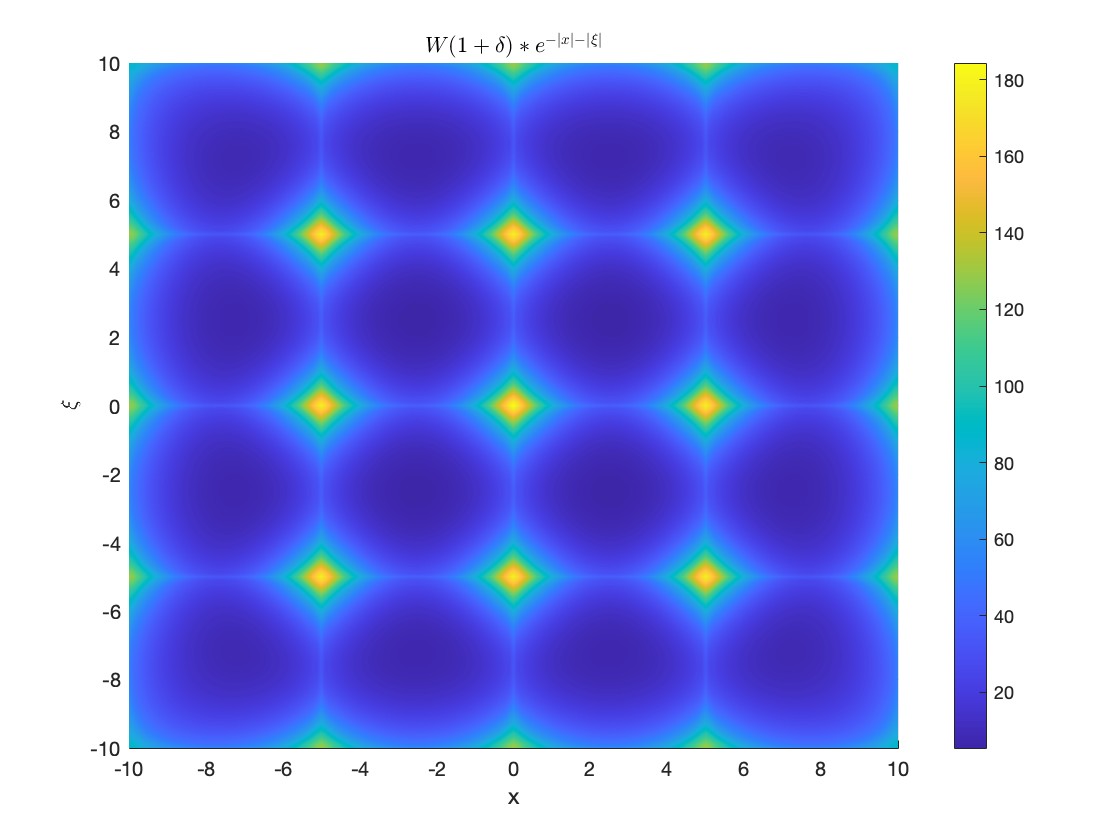}\includegraphics[width=0.5\textwidth]{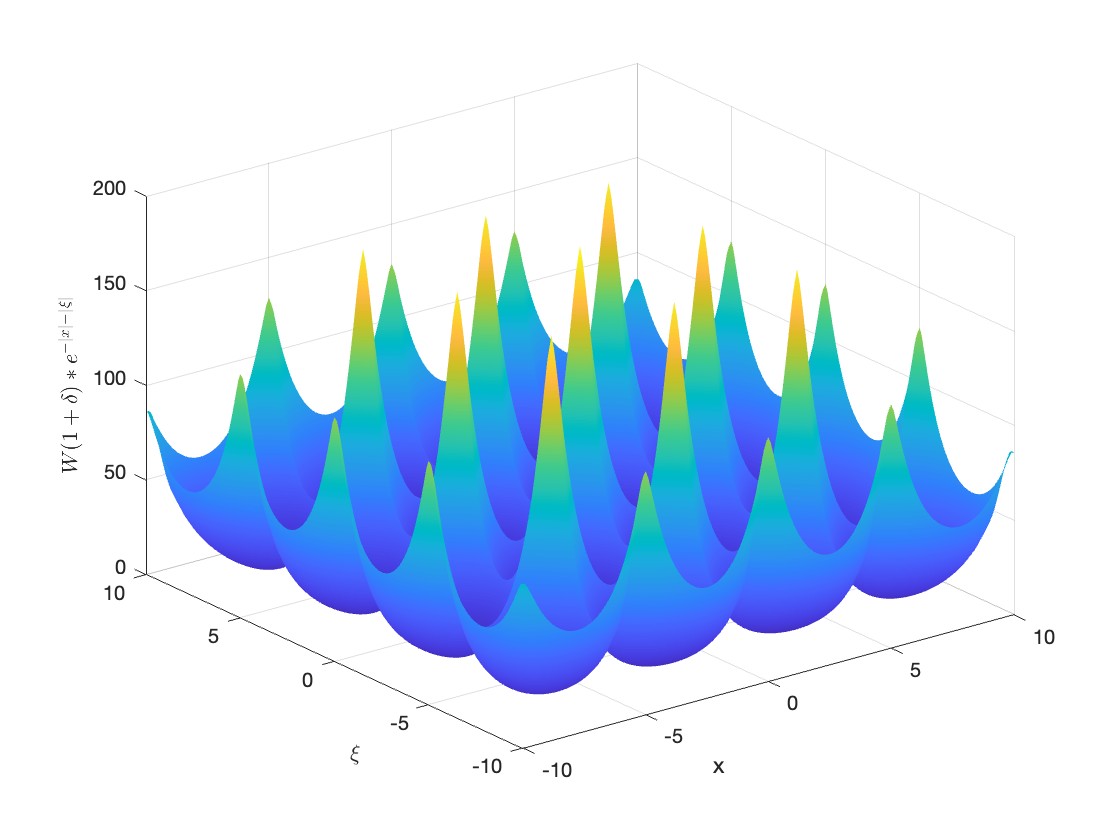}
		\caption{2D and 3D representations of $W(1+\delta)\ast e^{-|x|-|\xi|}$.}
		\label{fig:3}
	\end{figure}

	Figure \ref{fig:3} displays both the 2D and 3D graphic representations of $W(1+\delta)\ast e^{-|x|-|\xi|}$, showing the prominence of ghost frequencies. 
	
	\begin{figure}
		\includegraphics[width=0.5\textwidth]{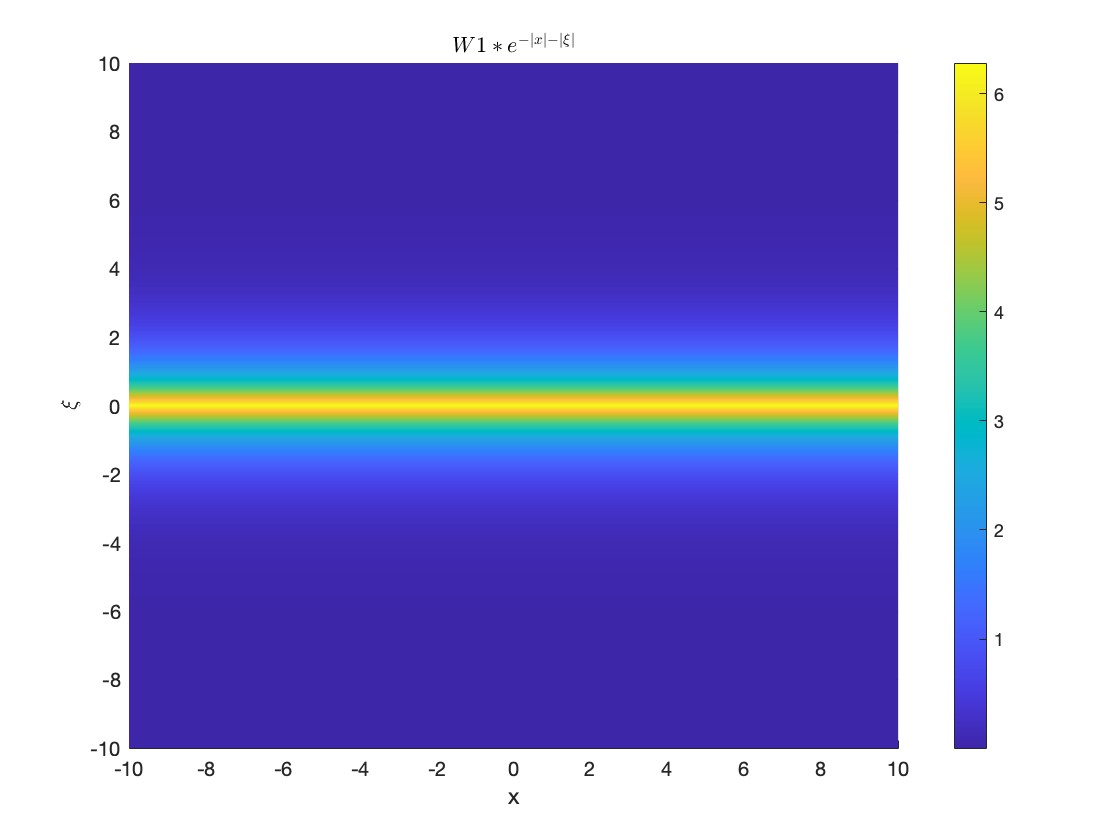}\includegraphics[width=0.5\textwidth]{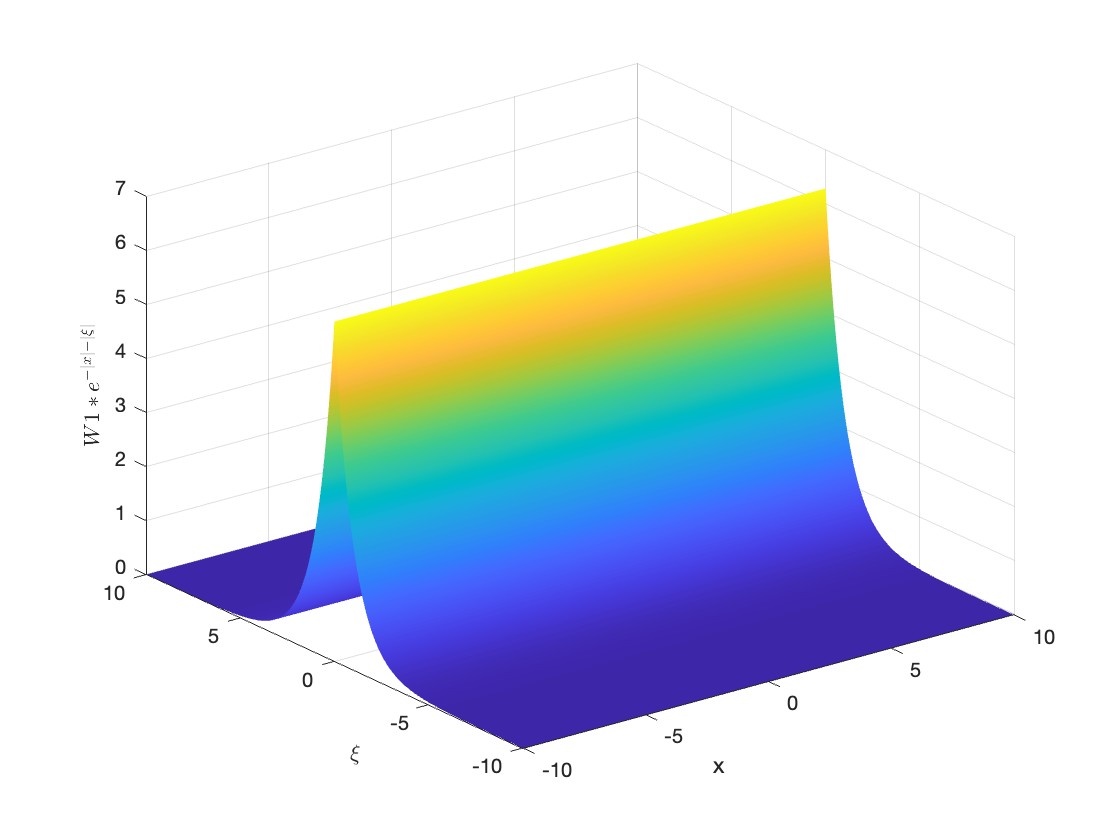}
		\caption{2D and 3D representations of $W1\ast e^{-|x|-|\xi|}$.}
		\label{fig:4}
	\end{figure}
		\begin{figure}
		\includegraphics[width=0.5\textwidth]{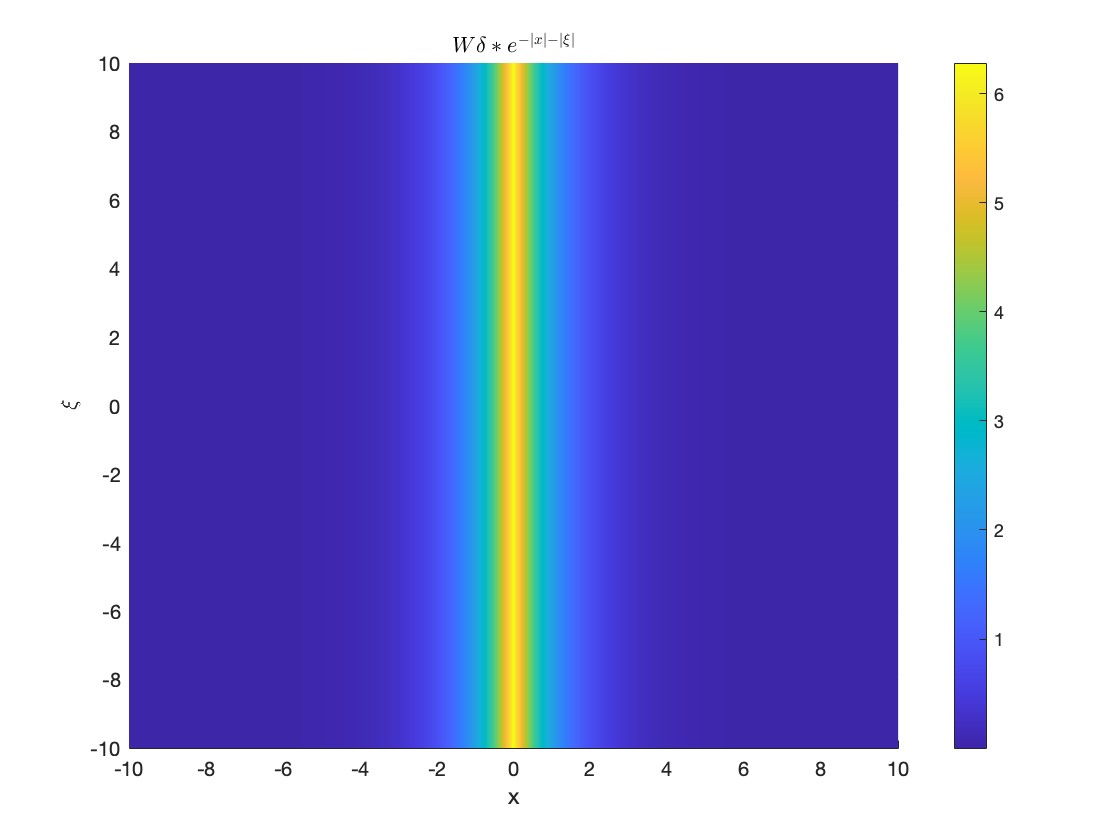}\includegraphics[width=0.5\textwidth]{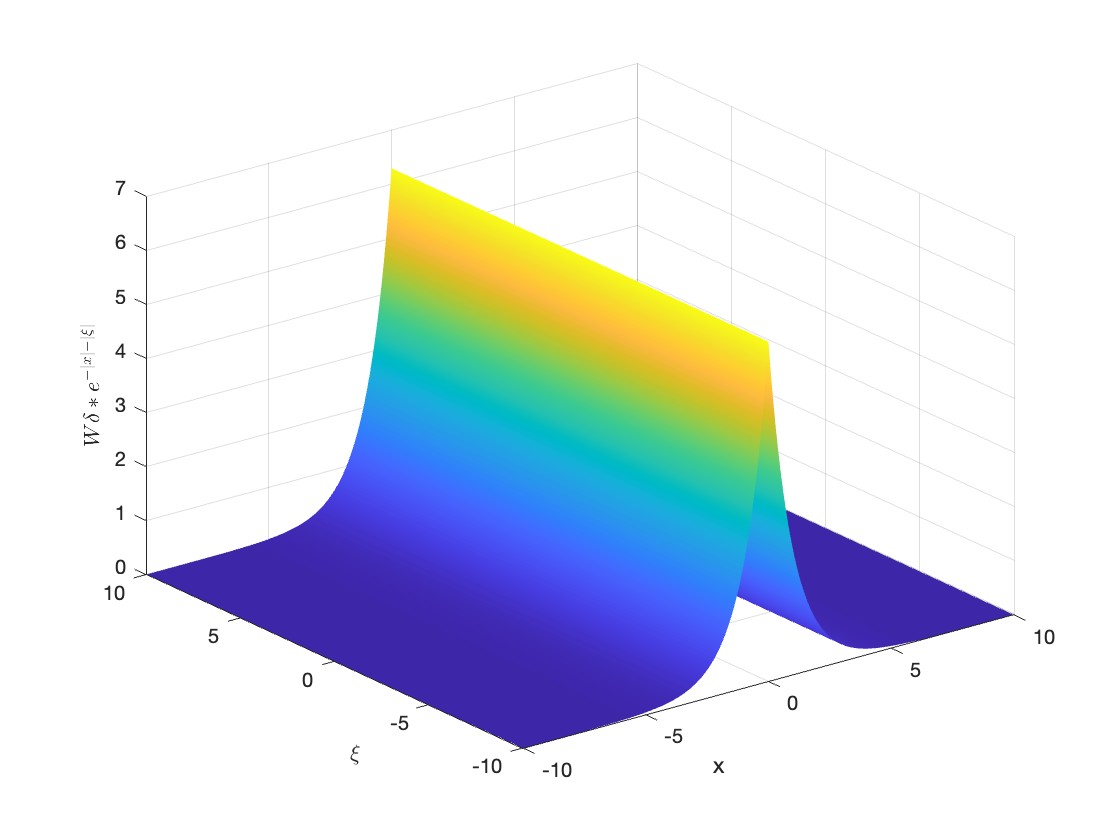}
		\caption{2D and 3D representations of $W\delta\ast e^{-|x|-|\xi|}$.}
		\label{fig:5}
	\end{figure}
	
\end{appendix}

\end{document}